\documentclass[a4paper,11pt]{article}
\usepackage[english]{babel}
\usepackage{amsmath,amssymb,amsthm,graphicx,enumerate,mathrsfs,bbm,xcolor,latexsym, dsfont}

\usepackage[pagebackref, draft=false]{hyperref}
\usepackage[top=3cm, bottom=3cm, left=2.3cm, right=2.3cm]{geometry}
\usepackage{breqn}

\newcommand{\assign}{:=}
\newcommand{\cdummy}{\cdot}
\newcommand{\dD}{\mathrm{D}}
\newcommand{\dd}{\mathrm{d}}
\newcommand{\nobracket}{}

\newcommand{\tmop}[1]{\ensuremath{\operatorname{#1}}}
\newcommand{\tmtextit}[1]{{\itshape{#1}}}
\newenvironment{enumerateroman}{\begin{enumerate}[i.] }{\end{enumerate}}
\newtheorem{theorem}{Theorem}[section]
\newtheorem{corollary}[theorem]{Corollary}
\newtheorem{definition}[theorem]{Definition}
\newtheorem{lemma}[theorem]{Lemma}
\newtheorem{remark}[theorem]{Remark}
\newtheorem{proposition}[theorem]{Proposition}

\newcommand{\R}{\mathbb{R}}
\newcommand{\Z}{\mathbb{Z}}
\newcommand{\N}{\mathbb{N}}
\newcommand{\T}{\mathbb{T}}
\newcommand{\E}{\mathbb{E}}
\renewcommand{\P}{\mathbb{P}}

\newcommand{\1}{\mathds{1}}

\newcommand{\CC}{\mathscr{C}}
\newcommand{\CD}{\mathscr{D}}

\newcommand{\CS}{\mathscr{S}}
\newcommand{\CF}{\mathscr{F}}
\newcommand{\CH}{\mathscr{H}}

\newcommand{\CQ}{\mathcal{Q}}

\newcommand{\CW}{\mathscr{W}}

\begin{document}

\title{Energy solutions of KPZ are unique}
\author{
  Massimiliano Gubinelli\\
  Hausdorff Center for Mathematics\\
   \& Institute for Applied Mathematics\\
   Universit{\"a}t Bonn \\
  \texttt{gubinelli@iam.uni-bonn.de}
  \and
  Nicolas Perkowski\thanks{Financial support by the DFG via Research Unit FOR 2402 is gratefully acknowledged.} \\
  Institut f\"ur Mathematik \\
  Humboldt--Universit\"at zu Berlin \\
  \texttt{perkowsk@math.hu-berlin.de}
}

\maketitle

\begin{abstract}
   The Kardar-Parisi-Zhang (KPZ) equation is conjectured to universally describe the fluctuations of weakly asymmetric interface growth. Here we provide the first intrinsic well-posedness result for the KPZ equation on the real line by showing that its \emph{energy solutions} as introduced by Gon\c calves and Jara in~\cite{bib:goncalvesJara} and refined in~\cite{bib:gubinelliJara} are unique. Together with the convergence results of~\cite{bib:goncalvesJara} and many follow-up papers this establishes the weak KPZ universality conjecture for a wide class of models. Our proof builds on an observation of Funaki and Quastel~\cite{bib:funakiQuastel}. A remarkable consequence is that the energy solution to the KPZ equation is not equal to the Cole-Hopf solution, but it involves an additional drift $t/12$. 
\end{abstract}

\section{Introduction}

The aim of this paper is to establish the well-posedness of the martingale problem for the stationary conservative Stochastic Burgers Equation (SBE) on $\mathbb{R}$,
\begin{equation}\label{eq:SBE-intro}
   \dd u_t = \nu \Delta u_t \dd t + \lambda \partial_x u^2_t \dd t + \sqrt{D} \partial_x \dd W_t ,
\end{equation}
where $u \colon
\mathbb{R}_+ \times \mathbb{R} \rightarrow \mathbb{R}$ is a  continuous process in $t$ taking values in the space of (Schwartz) distributions over $\mathbb{R}$, $\lambda \in \mathbb{R}$, $\nu, D > 0$, and $W$ is a space-time white noise. A direct consequence will be the well-posedness of the martingale problem for the quasi-stationary Kardar--Parisi--Zhang (KPZ) equation 
\begin{equation}\label{eq:KPZ-intro}
   \dd h_t = \Delta h_t \dd t + \lambda ((\partial_x h_t)^2 - \infty) \dd t + \sqrt{2} \dd W_t,
\end{equation}
where  $h \colon
\mathbb{R}_+ \times \mathbb{R} \rightarrow \mathbb{R}$ is a continuous process and for any $t$ the law of $h_t(\cdot) - h_t(0)$ is a two-sided Brownian motion on $\mathbb{R}$. The SBE describes the evolution of the weak derivative $u_t(x) = \partial_x h_t(x)$ of the solution to the KPZ equation $h$. Our uniqueness proof also establishes that $h$ is related to the  solution of the linear multiplicative Stochastic Heat equation (SHE),
\begin{equation}\label{eq:SHE-intro} \dd Z_t = \Delta Z_t \dd t + \sqrt{2} \lambda Z_t
  \dd W_t,
\end{equation}
by the Cole--Hopf transformation 
\begin{equation}\label{eq:CH-intro}
   h_t = \lambda^{- 1} \log Z_t + \frac{\lambda^3}{12} t, \qquad t \ge 0.
\end{equation}
The stochastic heat equation allows a formulation via standard It\^o calculus and martingale, weak, or mild solutions in suitable weighted spaces of continuous adapted processes. The stochastic Burgers equation and the KPZ equation, on the other hand, cannot be studied in standard spaces due to the fact that the non-linearity is ill-defined, essentially because the trajectories of the solutions do not possess enough spatial regularity. Indeed, solutions of the  KPZ equation are of H\"older regularity less than $1/2$ in space, so a priori the point-wise square of their derivatives cannot be defined. 

Despite this mathematical difficulty, the KPZ equation is expected to be a faithful description of the large scale properties of one-dimensional growth phenomena. This was the original motivation which led Kardar, Parisi and Zhang~\cite{bib:kpz} to study the equation and both experimental and theoretical physics arguments have, since then, confirmed their analysis. The rigorous study of the KPZ equation and its relation with the SHE has started with the work of Bertini and Giacomin~\cite{bib:bertiniGiacomin} on the scaling limit of the weakly asymmetric exclusion process (WASEP). Starting from this discrete Markov process on ${\{0,1\}}^\Z$ and performing a suitable space-time rescaling and recentering, they were able to prove that its density fluctuation field converges to a random field $u$ which is linked to the solution of the SHE by the Cole--Hopf transformation~\eqref{eq:CH-intro}. Incidentally they had to add exactly the strange $1/12$ drift in order to establish their result. Their work clarifies that any physically relevant notion of solution to the (still conjectural) equations~\eqref{eq:SBE-intro} and~\eqref{eq:KPZ-intro} needs to be transformed to the SHE by the Cole--Hopf transformation and also that the SBE should allow the law of the space white noise as invariant measure. A priori these insights are of little help in formulating KPZ/SBE, since given a solution to the SHE $Z$ it is not possible to apply It\^o's formula to $\lambda^{-1} \log Z$, and in particular the inverse Cole--Hopf transformation is ill-defined. It should be noted that the main difficulty of equations~\eqref{eq:SBE-intro} and~\eqref{eq:KPZ-intro} lies in the spatial irregularity and that no useful martingales in the space-variable are known, a fact which prevents an analysis via It\^o's stochastic integration theory. Moreover, the convergence result of~\cite{bib:bertiniGiacomin} relies strongly on the particular structure of the WASEP and does not  have many generalizations because most models behave quite badly under exponentiation (Cole--Hopf transformation); see~\cite{bib:demboTsai, bib:corwinTsai, bib:corwinShenTsai, bib:labbe} for examples of models that do admit a useful Cole--Hopf transformation.

After the work of Bertini and Giacomin there have been various attempts to study the SBE via Gaussian analysis tools taking into account the necessary invariance of the space white noise. A possible definition based on the Wick renormalized product associated to the driving space-time white noise has been ruled out because it lacks the properties expected from the physical solution~\cite{bib:chan}. Assing~\cite{bib:assing} has been the first, to our knowledge, to attempt a martingale problem formulation of the SBE. He defines a formal infinite-dimensional generator for the process, essentially as a quadratic form with dense domain, but is not able to prove its closability. The singular drift, which is ill-defined pointwise, make sense as a distribution on the Gaussian Hilbert space associated to the space white noise, however this distributional nature prevents to identify a suitable domain for the formal generator.

The martingale problem approach has been subsequently developed by Gon\c{c}alves and Jara~\cite{bib:gjEnergy, bib:goncalvesJara}\footnote{The paper~\cite{bib:goncalvesJara} is the revised published version of~\cite{bib:gjEnergy}.}. Their key insight is that while the drift in~\eqref{eq:SBE-intro} is difficult to handle in a Markovian picture (that is, as a function on the state space of the process) it makes perfect sense in a path-wise picture. They proved in particular that a large class of particle systems (which generalize the WASEP studied by Bertini and Giacomin) have fluctuations that subsequentially converge to random fields $u$ which are solutions of a generalized martingale problem for~\eqref{eq:SBE-intro} where the singular non-linear drift $\partial_x (u(t,x)^2)$ is a well defined  \emph{space-time} distributional random field. Avoiding to describe a Markovian generator for the process, they manage to introduce an auxiliary process which plays the same role in the formulation of the martingale problem. Subsequent work of Jara and Gubinelli~\cite{bib:gubinelliJara} gave a different definition of the martingale problem via a forward-backward description. The solution of the martingale problem is a Dirichlet process, that is the sum of a martingale and a finite variation process. This property and the forward-backward decomposition of the drift are reminiscent of Lyons--Zheng processes and in general of the theory of Markov processes described by Dirichlet forms, however a complete understanding of the matter is at the moment not well developed and the martingale problem formulation avoids the subtleties of the Markovian setting. Gon\c{c}alves and Jara called the solutions of this generalized martingale problem \emph{energy solutions} for the SBE/KPZ equation. 

Following~\cite{bib:goncalvesJara} it has been shown for a variety of models that their fluctuations subsequentially converge to energy solutions of the KPZ equation or the SBE, for example for zero range processes and kinetically constrained exclusion processes in~\cite{bib:goncalvesJaraSethuraman}, various exclusion processes in~\cite{bib:goncalvesSimon, bib:francoGoncalvesSimon, bib:blondelGoncalvesSimon, bib:goncalvesJaraLongrange}, interacting Brownian motions in~\cite{bib:diehlGubinelliPerkowski}, and Hairer-Quastel type SPDEs in~\cite{bib:gubinelliHQ}. This is coherent with the conjecture that the SBE/KPZ equation describes the universal behavior of a wide class of conservative dynamics or interface growth models in the particular limit where the asymmetry is ``small'' (depending on the spatial scale), the so called \emph{weak KPZ universality conjecture}, see~\cite{bib:corwin, bib:quastel, bib:quastelSpohn, bib:spohnSurvey}. In order to fully establish the conjecture for the models above, the missing step was a proof of uniqueness of energy solutions. This question remained open for some time during which it was not clear if the notion is strong enough to guarantee uniqueness or if it is too weak to expect well-posedness. Here we present a proof of uniqueness for the refined energy solutions of~\cite{bib:gubinelliJara}, on the full line and on the torus, thereby finally establishing the well-posedness of the martingale problem and its expected relation with the SHE via the Cole--Hopf transform. The proof follows the strategy developed by Funaki and Quastel in~\cite{bib:funakiQuastel}, namely we map a mollified energy solution to the stochastic heat equation via the Cole--Hopf transform and use a version of the Boltzmann--Gibbs principle to control the various error terms arising from the transformation and to derive the relation~\eqref{eq:CH-intro} in the limit as we take the mollification away. A direct corollary of our results is the proof of the weak KPZ universality conjecture for all the models in the literature which have been shown to converge to energy solutions.
 
Shortly after the introduction of energy solutions the fundamental work~\cite{bib:hairerKPZ} of Hairer on the KPZ equation appeared, where he established a path-wise notion of solution using Lyons' theory of rough paths to provide a definition of the non-linear term as a continuous bilinear functional on a suitable Banach space of functions. Existence and uniqueness were then readily established by fixed point methods. This breakthrough developed into a general theory of singular SPDEs, Hairer's theory of regularity structures~\cite{bib:hairerRegularity}, which provides the right analytic setting to control the singular terms appearing in stochastic PDEs like the SBE/KPZ equation, their generalizations, but also in other important SPDEs like the stochastic Allen--Cahn equation in dimensions $d=2,3$ and the (generalized) parabolic Anderson model in $d=2,3$. The work of the authors of this paper together with P. Imkeller on the use of paradifferential calculus~\cite{bib:paracontrolled} and the work of Kupiainen based on renormalization group (RG) techniques~\cite{bib:kupiainen, bib:kupiainenMarcozzi} opened other, alternative ways to tackle singular SPDEs. All these approaches have in common that they control the a priori ill-defined nonlinearities in the equation using path-wise (deterministic) arguments. However, from the point of view of the weak KPZ universality conjecture the path-wise approach is difficult to use and, for now, there are only few convergence results using either regularity structures, paracontrolled distributions or RG techniques~\cite{bib:hairerQuastel, bib:hairerShen, bib:kpzReloaded, bib:hoshino}. The martingale approach has the advantage that it is easy to implement, especially starting from discrete particle systems which often do not have the semi-linear structure that is at the base of the path-wise theories.

The main limitation of the martingale approach to the SBE/KPZ equation is that currently it works only at stationarity. Using tools from the theory of hydrodynamic limits it seems possible to extend the results to initial conditions with small relative entropy with respect to the stationary measure. However, this has not been done yet and dealing with even more singular initial conditions  is a completely open problem. On the other hand, with energy solutions it is relatively easy to work on the real line, while in the path-wise approach this requires dealing with weighted function spaces and the question of uniqueness seems still not clear.

To summarize, the main contribution of the present paper is a proof of uniqueness of energy solutions (in the refined formulation of Jara and Gubinelli~\cite{bib:gubinelliJara}) on the real line and on the torus. We start by introducing the notion of solution and the space of trajectories where solutions live in Section~\ref{sec:energySolutions}. Subsequently we discuss in Section~\ref{sec:additive} several key estimates available in this space, estimates which allow to control a large class of additive functional. After these preliminaries we show in Section~\ref{sec:proof} how to implement the Cole--Hopf transformation at the level of energy solutions and, by a careful control of some error terms, how to establish the It\^o formula which proves the mapping from the SBE to the SHE. Using the uniqueness for the SHE we conclude the uniqueness of energy solutions. In Appendix~\ref{app:periodic} we add some details on how to modify the proof to deal with the case of periodic boundary conditions.

\paragraph{Notation} The Schwartz space on $\R^d$ is denoted with $\CS(\R^d)$ and its dual $\CS'(\R^d)$ is the space of tempered distributions. The notation $\CD'(\R^d)$ refers to the distributions on $\R^d$, the dual space of $C^\infty_c(\R^d)$. The Fourier transform of $u \in \CS'(\R^d)$ is denoted with $\hat u = \CF u$, and we use the normalization $\hat u (\xi) = \int_{\R^d} e^{2\pi i \xi \cdot x} u(x) \dd x$. For $\alpha \in \R$ we use the following slightly unusual (but of course equivalent to the usual) norm for the space $H^\alpha(\R^d)$
\[
   H^\alpha(\R^d) \assign \left\{u \in \CS'(\R^d): \|u\|_{H^\alpha(\R^d)}^2 = \int_{\R^d} |\hat{u}(\xi)|^2 (1 + |2\pi \xi|^{2\alpha}) < \infty \right\}.
\]

Throughout we work with the quadratic variation in the sense of Russo and Vallois~\cite{bib:russoVallois}: A real-valued stochastic process $(X_t)_{t \geqslant 0}$ has quadratic variation $([X]_t)_{t \geqslant 0}$ if
\[
   [X]_t = \lim_{\varepsilon \to 0} \int_0^t \frac{1}{\varepsilon} (X_{s+\varepsilon} - X_s)^2 \dd s,
\]
where the convergence is uniform on compacts in probability. If $X$ is a continuous semimartingale, then $[X]$ is nothing but its semimartingale quadratic variation. Despite the fact that we deal with continuous processes we use the notation $[\cdot]$ for the quadratic variation, because $\langle \cdot, \cdot \rangle$ will be reserved for the inner products in various Hilbert spaces.

\section{Controlled processes and energy solutions}\label{sec:energySolutions}

\subsection{Burgers equation}

In this section we follow Gon\c calves and Jara~\cite{bib:goncalvesJara} and Gubinelli and Jara~\cite{bib:gubinelliJara} in
defining stationary energy solutions to the stochastic Burgers equation $u \colon
\mathbb{R}_+ \times \mathbb{R} \rightarrow \mathbb{R}$,
\begin{equation}\label{eq:SBE-param}
   \dd u_t = \nu \Delta u_t \dd t + \lambda \partial_x u^2_t \dd t + \sqrt{D} \partial_x \dd W_t ,
\end{equation}
where $\lambda \in \mathbb{R}$, $\nu, \sqrt{D} > 0$, and $W$ is a space-time white noise. Recall that from a probabilistic point of
view the key difficulty in making sense of~(\ref{eq:SBE-param}) is that we expect
the law of (a multiple of) the white noise on $\R$ to be invariant under the dynamics, but the square
of the white noise can only be defined as a Hida distribution and not as a
random variable. To overcome this problem we first introduce a class of
processes $u$ which at fixed times are distributed as the white noise but for which the nonlinear term $\partial_x u^2$ is defined as
a space-time distribution. In this class of processes it then makes sense to
look for solutions of the Burgers equation~(\ref{eq:SBE-param}).

If $(\Omega, \mathcal{F}, (\mathcal{F}_t)_{t \geqslant 0}, \mathbb{P})$ is a
filtered probability space, then an adapted process $W$ with trajectories in
$C \left( \mathbb{R}_+, \CS' (\mathbb{R}) \right)$ is called a \emph{space-time
white noise} on that space if for all $\varphi \in \CS (\mathbb{R})$ the
process $(W_t (\varphi))_{t \geqslant 0}$ is a Brownian motion in the
filtration $(\mathcal{F}_t)_{t \geqslant 0}$ with variance $\mathbb{E} [W_t (\varphi)^2] = t
\| \varphi \|^2_{L^2 (\mathbb{R})}$ for all $t \geqslant 0$. A \emph{(space) white
noise with variance $\sigma^2$} is a random variable $\eta$ with values in $\CS' (\mathbb{R})$, such
that $(\eta (\varphi))_{\varphi \in \CS (\mathbb{R})}$ is a centered Gaussian process with covariance $\mathbb{E} [\eta (\varphi) \eta (\psi)] = \sigma^2
\langle \varphi, \psi \rangle_{L^2 (\mathbb{R})}$. If $\sigma = 1$ we simply call $\eta$ a white noise. Throughout we write $\mu$ for the law of the white noise on $\CS' (\mathbb{R})$.

\begin{definition}[Controlled process]
  Let $\nu, D > 0$, let $W$ be a space-time white noise on the filtered
  probability space $(\Omega, \mathcal{F}, (\mathcal{F}_t)_{t \geqslant 0},
  \mathbb{P})$, and let $\eta$ be a $\mathcal{F}_0$-measurable space white
  noise. Denote with $\mathcal{Q}_{\nu, D} (W, \eta)$ the space of pairs $(u,
  \mathcal{A})$ of adapted stochastic processes with trajectories in $C \left(
  \mathbb{R}_+, \CS' (\mathbb{R}) \right)$ such that
  \begin{itemize}
    \item[i)] $u_0 = \sqrt{D / (2 \nu)} \eta$ and the law of $u_t$ is that of a white
    noise with variance $D / (2 \nu)$ for all $t \geqslant 0$;
    
    \item[ii)] for any test function $\varphi \in \CS (\mathbb{R})$ the
    process $t \mapsto \mathcal{A}_t (\varphi)$ is almost surely of zero
    quadratic variation, satisfies $\mathcal{A}_0 (\varphi) = 0$, and the pair
    $(u (\varphi), \mathcal{A} (\varphi)) $ solves the equation
    \begin{equation}
      \label{eq:controlled-decomposition} u_t (\varphi) = u_0(\varphi) + \nu
      \int_0^t u_s (\Delta \varphi) \dd s +\mathcal{A}_t (\varphi) -
      \sqrt{D} W_t (\partial_x \varphi), \qquad t \geqslant 0 ;
    \end{equation}
    \item[iii)] for any $T > 0$ the time-reversed processes $\hat{u}_t = u_{T
    - t}$, $\hat{\mathcal{A}}_t = - (\mathcal{A}_T -\mathcal{A}_{T -
    t})$ satisfy
    \[ \hat{u}_t (\varphi) = \hat{u}_0 (\varphi) + \nu \int_0^t \hat{u}_s
       (\Delta \varphi) \dd s + \hat{\mathcal{A}}_t (\varphi) - \sqrt{D}
       \hat{W}_t (\partial_x \varphi), \qquad t \in [0, T], \varphi \in \CS(\R), \]
    where $\hat{W}$ is a space-time white noise in the filtration generated by
    $(\hat{u}, \hat{\mathcal{A}})$.
  \end{itemize}
  If there exist a space-time white noise $W$ and a white noise $\eta$ such
  that $(u, \mathcal{A}) \in \mathcal{Q}_{\nu, D} (W, \eta)$, then we simply
  write $(u, \mathcal{A}) \in \mathcal{Q}_{\nu, D}$. For $\nu=1$ and $D=2$  we omit the parameters in the notation and write $(u, \mathcal{A}) \in \mathcal{Q}(W, \eta)$ respectively $(u, \mathcal{A}) \in \mathcal{Q}$.
\end{definition}

We will see that $\CQ_{\nu,D}(W,\eta)$ contains the probabilistically strong solution to~\eqref{eq:SBE-param}, while $\CQ_{\nu,D}$ is the space in which to look for probabilistically weak solutions.

Controlled processes were first introduced in~\cite{bib:gubinelliJara} on the
circle, and the definition on the real line is essentially the same. For
$\mathcal{A}= 0$ the process $(X, 0) \in \mathcal{Q}_{\nu,D}$ is the stationary
Ornstein--Uhlenbeck process. It is the unique-in-law solution to the SPDE
\[
   \dd X_t = \nu \Delta X_t \dd t + \sqrt{D} \partial_x \dd W_t
\]
with initial condition $u_0 \sim \sqrt{D/(2\nu)} \eta$. Allowing $\mathcal{A} \neq 0$ has the
intuitive meaning of considering perturbations of the Ornstein--Uhlenbeck
process with antisymmetric drifts of zero quadratic variation. In this sense
we say that a couple $(u, \mathcal{A}) \in \mathcal{Q}_{\nu,D}$ is a process
\tmtextit{controlled} by the Ornstein--Uhlenbeck process.

As we will see below, for controlled processes we are able to construct some
interesting additive functionals. In particular, for any controlled process
the Burgers drift makes sense as a space-time distribution:

\begin{proposition}\label{prop:burgers drift}
  Let $(u, \mathcal{A}) \in \mathcal{Q}_{\nu,D}$, let $\rho
  \in L^1 (\mathbb{R}) \cap L^2 (\mathbb{R})$ with $\int_{\mathbb{R}} \rho
  (x) \dd x = 1$, and write $\rho^N \assign N \rho (N \cdummy)$ for $N \in
  \mathbb{N}$. Then for all $\varphi \in \CS (\mathbb{R})$ the process
  \[ \int_0^t (u_s \ast \rho^N)^2 (- \partial_x \varphi) \dd s, \qquad t
     \geqslant 0, \]
  converges uniformly on compacts in probability to a limiting process that we
  denote with
  \[ \int_0^t \partial_x u_s^2 \dd s (\varphi) , \qquad t \geqslant 0. \]
  As the notation suggests, this limit does not depend on the function $\rho$.
\end{proposition}

The proof will be given in Section~\ref{sec:nonlinearity} below. Note that the convolution $u_s \ast \rho^N$ is well-defined for $\rho^N \in L^2 (\mathbb{R}) \cap L^1 (\mathbb{R})$ and not only for $\rho^N \in \CS
(\mathbb{R})$ because at fixed times $u_s$ is a white noise.

Now we can define what it means for a controlled process to solve the
stochastic Burgers equation.

\begin{definition}
  Let $W$ be a space-time white noise on $(\Omega, \mathcal{F}, (\mathcal{F}_t)_{t \geqslant 0}, \mathbb{P})$, let
  $\eta$ be a $\mathcal{F}_0$-measurable space white noise, and let $(u,
  \mathcal{A}) \in \mathcal{Q}_{\nu,D} (W, \eta)$. Then $u$ is called a
  \emph{strong stationary solution} to the stochastic Burgers equation
  \begin{equation}\label{eq:SBE strong}
     \dd u_t = \nu \Delta u_t \dd t + \lambda \partial_x u_t^2 \dd t + \sqrt{D} \partial_x \dd W_t, \qquad u_0 = \sqrt{\frac{D}{2\nu}}\eta,
  \end{equation}
  if $\mathcal{A} (\varphi) = \lambda \int_0^{\cdummy} \partial_x u_s^2
   \dd s(\varphi)$ for all $\varphi \in \CS(\R)$. If $(u, \mathcal{A}) \in \mathcal{Q}_{\nu,D}$ and $\mathcal{A} (\varphi) = \lambda
  \int_0^{\cdummy} \partial_x u_s^2  \dd s(\varphi)$ for all $\varphi \in
  \CS(\R)$, then $u$ is called an \emph{energy solution} to~(\ref{eq:SBE
  strong}).
\end{definition}

The notion of energy solutions was introduced in a slightly weaker formulation
by Gon{\c c}alves and Jara in~\cite{bib:goncalvesJara} and the formulation here is
due to~\cite{bib:gubinelliJara}. Note that strong stationary solutions correspond
to probabilistically strong solution, while energy solutions are
probabilistically weak in the sense that we do not fix the probability space
and the noise driving the equation. Our main result is the following theorem
which establishes the strong and weak uniqueness of our solutions.

\begin{theorem}\label{thm:burgers-uniqueness}
  Let $W$ be a space-time white noise on the filtered probability space
  $(\Omega, \mathcal{F}, (\mathcal{F}_t)_{t \geqslant 0}, \mathbb{P})$ and
  let $\eta$ be a $\mathcal{F}_0$-measurable space white noise. Then the
  strong stationary $u$ solution to
  \begin{equation}\label{eq:thm-SBE}
     \dd u_t = \nu \Delta u_t \dd t + \lambda \partial_x  u_t^2 \dd t + \sqrt{D} \partial_x \dd W_t, \qquad u_0 =  \sqrt{\frac{D}{2\nu}} \eta,
  \end{equation}
  is unique up to indistinguishability. Moreover, for $\lambda \neq 0$ we have
  $u = (\nu/\lambda) \partial_x \log Z^{(\sigma)}$, where the derivative is taken in the
  distributional sense and $Z^{(\sigma)}$ is the unique solution to the linear multiplicative stochastic
  heat equation
  \begin{equation}\label{eq:thm-SHE} \dd Z^{(\sigma)}_t = \nu \Delta Z^{(\sigma)}_t \dd t + \frac{\lambda \sqrt{D}}{\nu} Z^{(\sigma)}_t
    \dd W_t, \qquad Z^{(\sigma)}_0 (x) = e^{\frac{\lambda}{\nu} \sqrt{\frac{D}{2\nu}} \eta (\Theta^{(\sigma)}_x)},
  \end{equation}
  and where $\Theta^{(\sigma)}_x =\1_{(- \infty, x]} - \int_{\cdummy}^{\infty}
  \sigma (y) \dd y$ for an arbitrary $\sigma \in C^{\infty}_c (\mathbb{R})$ with $\sigma \geqslant 0$ and
  $\int_{\mathbb{R}} \sigma (x) \dd x = 1$. Consequently, any two energy solutions of~(\ref{eq:thm-SBE}) have the
  same law.
\end{theorem}

The proof will be given in Section~\ref{sec:map-SHE} below.

\begin{remark}
  As far as we are aware, this is the first time that existence and
  uniqueness for an intrinsic notion of solution to the Burgers equation on the
  real line is obtained. The results of~\cite{bib:hairerKPZ, bib:kpzReloaded, bib:kupiainenMarcozzi} are
  restricted to the circle $\mathbb{T}$. In~\cite{bib:hairerLabbe}
  the linear multiplicative heat equation is solved on the real line using
  regularity structures, and by the strong maximum principle of~\cite{bib:cannizzaroGassiatFriz} the solution is strictly positive so in particular
  its logarithm is well-defined. Then it should be possible to show that the derivative of the
  logarithm is a modelled distribution and solves the stochastic Burgers equation in the
  sense of regularity structures. However, it is not at all obvious if and in which sense
  this solution is unique.
\end{remark}

\begin{remark}[Reduction to standard parameters]
   To simplify notation we will assume from now on that $\nu = 1$ and $D = 2$ which can be achieved by a simple transformation. Indeed, it is easy to see that $(u, \mathcal{A}) \in \mathcal{Q}_{\nu, D}(W,\eta)$ if and only if $(u^{\nu, D}, \mathcal{A}^{\nu, D}) \in \mathcal{Q}(W^{\nu,D},\eta)$, where
  \[ u^{\nu, D}_t (\varphi) \assign \sqrt{\frac{2 \nu}{D}}  u_{t / \nu} (\varphi), \qquad \mathcal{A}^{\nu, D}_t (\varphi) \assign
     \sqrt{\frac{2 \nu}{D}} \mathcal{A}_{t / \nu} (\varphi), \qquad W^{\nu,D}_t(\varphi) \assign \sqrt{\nu} W_{t/\nu}(\varphi)
  \]
  and $W^{\nu, D}$ is a space-time white noise. Moreover, $u$ is a strong stationary solution to~\eqref{eq:thm-SBE} if and only if $u^{\nu,D}$ is a strong stationary solution to
  \[
     \dd u^{\nu, D}_t = \Delta u^{\nu, D}_t \dd t + \lambda \sqrt{\frac{D}{2 \nu^3} } \partial_x (u^{\nu, D}_t)^2 \dd t +\sqrt{2} \partial_x \dd W^{\nu,D}_t, \qquad u^{\nu,D}_0 = \eta.
  \]
\end{remark}

\begin{remark}[Martingale problem]\label{rmk:martingale problem}Energy solutions can be understood as
  solutions to a martingale problem. Indeed, given a pair of stochastic
  processes $(u, \mathcal{A})$ with trajectories in $C \left( \mathbb{R}_+, \CS'(\R) \right)$ we
  need to check the following criteria to verify that $u$ is the unique-in-law
  energy solution to~(\ref{eq:thm-SBE}):
  \begin{enumerateroman}
    \item the law of $u_t$ is that of the white noise with variance $D/(2\nu)$ for all $t \geqslant 0$;
    
    \item for any test function $\varphi \in \CS (\mathbb{R})$ the process $t
    \mapsto \mathcal{A}_t (\varphi)$ is almost surely of zero quadratic
    variation, satisfies $\mathcal{A}_0 (\varphi) = 0$, and the pair $(u
    (\varphi), \mathcal{A} (\varphi)) $ solves the equation
    \[ u_t (\varphi) = u_0 (\varphi) + \nu \int_0^t u_s (\Delta \varphi) \dd s
       +\mathcal{A}_t (\varphi) + M_t (\varphi), \qquad t \geqslant 0, \]
    where $M (\varphi)$ is a continuous martingale in the filtration generated
    by $(u, \mathcal{A})$, such that $M_0 (\varphi) = 0$ and $M (\varphi)$
    has quadratic variation $[M (\varphi)]_t = D \| \partial_x \varphi
    \|_{L^2}^2 t$.
    
    \item for any $T > 0$ the time-reversed processes $\hat{u}_t = u_{T - t}$,
    $\hat{\mathcal{A}}_t = - (\mathcal{A}_T -\mathcal{A}_{T - t})$
    satisfy
    \[ \hat{u}_t (\varphi) = \hat{u}_0 (\varphi) + \nu  \int_0^t \hat{u}_s (\Delta
       \varphi) \dd s + \hat{\mathcal{A}}_t (\varphi) + \hat{M}_t
       (\varphi), \qquad t \in [0, T], \]
    where $\hat{M} (\varphi)$ is a continuous martingale in the filtration
    generated by $(\hat{u}, \hat{\mathcal{A}})$, such that $\hat{M}_0
    (\varphi) = 0$ and $\hat{M} (\varphi)$ has quadratic variation $[\hat{M}
    (\varphi)]_t = D \| \partial_x \varphi \|_{L^2}^2 t$.
    \item there exists $\rho \in L^1(\R) \cap L^2(\R)$ such that $\int_\R \rho(x) \dd x = 1$ and with $\rho^N \assign N \rho(N \cdot)$ we have
       \[
          \mathcal{A}_t(\varphi) = \lambda \lim_{N\to \infty} \int_0^t (u_s \ast \rho^N)^2 (- \partial_x \varphi) \dd s
       \]
       for all $t \ge 0$ and $\varphi \in \CS(\R)$.
  \end{enumerateroman}
\end{remark}

\begin{remark}[Different notions of energy solutions]
  In \cite{bib:goncalvesJara} a slightly weaker notion of energy solution was
  introduced. Roughly speaking Gon\c calves and Jara only made assumptions i., ii. and iv. of
  Remark~\ref{rmk:martingale problem} but did not consider the time reversal
  condition~iii. We do not know whether this weaker notion still gives rise to
  unique solutions. However, the convergence proof of~\cite{bib:goncalvesJara}
  easily gives all four assumptions i., ii., iii., iv. even if in that paper
  it is not explicitly mentioned. Indeed, the time-reversal condition iii. is
  satisfied on the level of the particle systems they consider, and
  it trivially carries over to the limit. Therefore, the combination of~\cite{bib:goncalvesJara} with our
  uniqueness result proves the weak KPZ universality conjecture for a wide class of particle systems.
\end{remark}

\begin{remark}[Relaxations]
  One can show that it suffices to assume
  $(u, \mathcal{A}) \in C \left( \mathbb{R}_+, \CD'(\R) \right)$ and to verify
  the conditions of Remark~\ref{rmk:martingale problem} for $\varphi \in
  C^{\infty}_c (\mathbb{R})$ instead of $\varphi \in \CS (\mathbb{R})$. For our proof of uniqueness we do need to handle test functions in $\CS(\R)$, but since all terms in the decomposition
  \[
     u_t(\varphi) = u_0 (\varphi) + \nu \int_0^t u_s(\Delta \varphi) \dd s + \lambda \int_0^t \partial_x u_s^{\diamond 2} \dd s (\varphi) + \sqrt{D} W_t(\varphi)
  \]
  come with good continuity estimates it follows from the above conditions that $u$ and $\mathcal{A}$ have trajectories in $C(\R_+, \CS'(\R))$ and satisfy the same equation also for test functions in $\CS(\R)$. It is even possible to only assume  that $(u(\varphi), \mathcal{A}(\varphi))$ is a family of continuous adapted stochastic process indexed by $\varphi \in C^\infty_c(\R)$ such that the conditions i.-iv. in Remark~\ref{rmk:martingale problem} hold up to a null set that may depend on $\varphi$. In that case we can find a version $(\tilde{u},\tilde{\mathcal A})$ with values in $C(\R_+, \CS'(\R))$ such that $\P(u(\varphi) = \tilde u(\varphi), \mathcal{A}(\varphi)=\tilde{\mathcal A}(\varphi)) = 1$ for all $\varphi \in C^\infty_c(\R)$. These relaxations may come in handy when proving
  the convergence of fluctuations of microscopic systems to the Burgers equation. 
\end{remark}

\subsection{KPZ equation}

Once we understand how to deal with the stochastic Burgers equation, it is not difficult to handle also the KPZ equation. Let $W$ be a space-time white noise on the filtered probability space $(\Omega, \mathcal{F}, (\mathcal{F}_t)_{t \geqslant 0}, \mathbb{P})$ and let $\chi$ be a $\mathcal{F}_0$-measurable random variable with values in $\CS'(\R)$ such that $\eta \assign \partial_x \chi$ is a space white noise. Then we denote with $\mathcal{Q}^{\mathrm{KPZ}} (W, \chi)$ the space of pairs $(h, \mathcal{B})$ of adapted
  stochastic processes with trajectories in $C ( \mathbb{R}_+, \CS' (\mathbb{R}) )$ which solve for all $\varphi \in \CS(\R)$ the equation
\begin{equation}\label{eq:KPZ-controlled}
   h_t(\varphi) = \chi(\varphi) + \int_0^t h_s (\Delta \varphi) \dd s + \mathcal B_t(\varphi) + \sqrt{2} W_t(\varphi),
\end{equation}  
 and are such that $\mathcal{B}_0(\varphi) = 0$, and for which $u \assign \partial_x h$, $\mathcal{A} \assign \partial_x \mathcal{B}$ satisfy $(u, \mathcal{A}) \in \CQ(W,\eta)$. Similarly we write $(h, \mathcal{B}) \in \CQ^{\mathrm{KPZ}}$ if~\eqref{eq:KPZ-controlled} holds and $(u, \mathcal A) \in \CQ$. In Proposition~\ref{prop:kpz drift} we show in fact the following generalization of Proposition~\ref{prop:burgers drift}: 
  If $(u, \mathcal{A}) \in \CQ$ and $(\rho^N) \subset L^1(\R) \cap L^2(\R)$ is an approximate identity (i.e. $\sup_N \| \rho^N
  \|_{L^1(\R)} < \infty$, $\widehat{\rho^N} (0) = 1$ for all $N$, $\lim_{N \rightarrow
  \infty} \widehat{\rho^N} (x) = 1$ for all $x$), then there exists a process
  $\int_0^{\cdummy} u_s^{\diamond 2} \dd s \in C \left( \R_+, \CS'(\R)
  \right)$, defined by
  \[ \left( \int_0^t u_s^{\diamond 2} \dd s \right) (\varphi) \assign
     \lim_{n \rightarrow \infty} \int_0^t [(u_s \ast \rho^N)^2 - \| \rho^N
     \|_{L^2(\R)}^2] (\varphi) \dd s, \]
  where the convergence takes place in $L^p(\P)$ uniformly on compacts and the limit does not depend on the approximate
  identity $(\rho^N)$. 

So we call $(h, \mathcal{B}) \in \CQ^{\mathrm{KPZ}}(W,\chi)$ a \emph{strong almost-stationary solution} to the KPZ equation
\begin{equation}\label{eq:KPZ}
   \dd h_t = \Delta h_t \dd t + \lambda ((\partial_x h_t)^2 - \infty) \dd t + \sqrt{2} \dd W_t,\qquad h_0 = \chi,
\end{equation}
if $\mathcal{B} = \lambda \int_0^\cdot u_s^{\diamond 2} \dd s$. Similarly we call $(h, \mathcal{B}) \in \CQ^{\mathrm{KPZ}}$ an \emph{energy solution} to~\eqref{eq:KPZ} if $\mathrm{law}(h_0) = \mathrm{law}(\chi)$ and $ \mathcal{B} = \lambda \int_0^\cdot u_s^{\diamond 2} \dd s$. The terminology ``almost-stationary'' comes from~\cite{bib:goncalvesJara} and it indicates that for fixed $t \geqslant 0$ the process $x \mapsto h_t(x)$ is always a two-sided Brownian motion, however the distribution of $h_t(0)$ may depend on time. The analogous result to Theorem~\ref{thm:burgers-uniqueness} is then the following.

\begin{theorem}\label{thm:kpz-uniqueness}
  Let $W$ be a space-time white noise on $(\Omega, \mathcal{F}, (\mathcal{F}_t)_{t \geqslant 0}, \mathbb{P})$ and
  let $\chi$ be a $\mathcal{F}_0$-measurable random variable with values in $\CS'(\R)$ such that $\partial_x \chi$ is a space white noise. Then the
  strong almost-stationary solution $h$ to
  \begin{equation}\label{eq:thm-KPZ}
     \dd h_t = \Delta h_t \dd t + \lambda ((\partial_x h_t)^2 - \infty) \dd t + \sqrt{2} \dd W_t, \qquad h_0 = \chi,
  \end{equation}
  is unique up to indistinguishability. Moreover, for $\lambda \neq 0$ we have
  \begin{equation}\label{eq:thm-KPZ-CH}
     h_t = \lambda^{- 1} \log Z_t + \frac{\lambda^3}{12} t, \qquad t \ge 0,
  \end{equation}
  where $Z$ is the unique solution to the linear multiplicative stochastic heat equation
  \begin{equation}\label{eq:thm-KPZ-SHE} \dd Z_t = \Delta Z_t \dd t + \sqrt{2} \lambda Z_t
    \dd W_t, \qquad Z_0 = e^{\lambda \chi}.
  \end{equation}
  Consequently, any two energy solutions of~(\ref{eq:thm-KPZ}) have the
  same law.
\end{theorem}

The proof will be given in Section~\ref{sec:map-SHE} below.

\begin{remark}
   It is maybe somewhat surprising that the energy solution $h$ to the KPZ equation is not equal to the Cole-Hopf solution $\tilde{h} \assign \lambda^{-1} \log Z$ but instead we have to add the drift $t \lambda^3/12$ to $\tilde {h}$ to obtain $h$. Remarkably, this drift often appears in results about the Cole-Hopf solution of the KPZ equation, for example in~\cite[Theorem~2.3]{bib:bertiniGiacomin} it has to be added to obtain the Cole-Hopf solution as scaling limit for the fluctuations of the height profile of the weakly asymmetric exclusion process (there the drift is $- t/24$ because Bertini and Giacomin consider different parameters for the equation). The same drift also appears in~\cite[Theorem~1.1]{bib:amirCorwinQuastel}, in the key formula~(4.17) of~\cite{bib:sasamotoSpohnExact}, and in~\cite[Theorem~1.1]{bib:funakiQuastel}.
   
   In~\cite[Theorem~3]{bib:goncalvesJara} it is claimed that the Cole-Hopf solution is an energy solution to the KPZ equation, and as we have seen this is not quite correct. The reason is that the proof in~\cite{bib:goncalvesJara} is based on the convergence result of~\cite{bib:bertiniGiacomin}, but they did not take the drift $-t/24$ into account which Bertini and Giacomin had to add to obtain the Cole-Hopf solution in the limit.
\end{remark}

\begin{remark}[Martingale problem]\label{rmk:martingale problem kpz}
  Given a pair of stochastic processes $(h, \mathcal{B})$ with trajectories in $C \left( \mathbb{R}_+, \CS'(\R) \right)$ we
  need to check the following criteria to verify that $h$ is the unique-in-law
  energy solution to~(\ref{eq:thm-KPZ}):
  \begin{enumerateroman}
    \item for all $\varphi \in \CS(\R)$ we have $\mathcal{B}_0(\varphi) = 0$ and
       \begin{equation*}
          h_t(\varphi) = \chi(\varphi) + \int_0^t h_s (\Delta \varphi) \dd s + \mathcal B_t(\varphi) + W_t(\varphi)
       \end{equation*}  
       for a continuous martingale $W(\varphi)$ starting in $0$ and with quadratic variation $[W(\varphi)]_t = 2 t \| \varphi\|_{L^2(\R)}^2$;
    \item the pair $u \assign \partial_x h$, $\mathcal{A} \assign \partial_x \mathcal{B}$ satisfies conditions i., ii., iii. in Remark~\ref{rmk:martingale problem};
    
    \item there exists $\rho \in L^1(\R) \cap L^2(\R)$ such that $\int_\R \rho(x) \dd x = 1$ and with $\rho^N \assign N \rho(N \cdot)$ we have
       \[
          \mathcal{B}_t(\varphi) = \lambda \lim_{N\to \infty} \int_0^t [(u_s \ast \rho^N)^2 - \| \rho^N \|_{L^2(\R)}^2 ](\varphi) \dd s
       \]
       for all $t \ge 0$ and $\varphi \in \CS(\R)$.
  \end{enumerateroman}
\end{remark}

\subsection{The periodic case}\label{sec:periodic}

It is also useful to have a theory for the periodic model $u \colon \R_+ \times \T \to \R$, where $\T = \R/\Z$ and
\begin{equation}\label{eq:SBE-periodic}
   \dd u_t =  \Delta u_t \dd t + \lambda \partial_x u^2_t \dd t + \sqrt{2} \partial_x \dd W_t ,\qquad u_0 = \eta,
\end{equation}
for a periodic space-time white noise $W$ and a periodic space white noise $\eta$. A periodic space-time white noise is a process $W$ with trajectories in $C(\R_+, \CS'(\T))$, where $\CS'(\T)$ are the (Schwartz) distributions on the circle, such that for all $\varphi \in C^\infty(\T)$ the process $(W_t(\varphi))_{t \geqslant 0}$ is a Brownian motion with variance $\E[|W_t(\varphi)|^2] = t \|\varphi\|_{L^2(\T)}^2$. A periodic space white noise is a centered Gaussian process $(\eta(\varphi))_{\varphi \in C^\infty(\T)}$ with trajectories in $\CS'(\T)$, such that for all $\varphi, \psi \in C^\infty(\T)$ we have $\E[\eta(\varphi) \eta(\psi)] = \langle \Pi_0 \varphi, \Pi_0 \psi \rangle_{L^2(\T)}$, where $\Pi_0 \varphi = \varphi - \int_{\T} \varphi(x) \dd x$ is the projection of $\varphi$ onto the mean-zero functions. The reason for setting the zero Fourier mode of $\eta$ equal to zero is that the stochastic Burgers equation is a conservation law and any solution $u$ to~\eqref{eq:SBE-periodic} satisfies $\widehat{u_t}(0) = \widehat{u_0}(0)$ for all $t \geqslant 0$, and therefore shifting $\hat{\eta}(0)$ simply results in a shift of $\widehat{u_t}(0)$ by the same value, for all $t \geqslant 0$. So for simplicity we assume $\hat{\eta}(0) = 0$. Controlled processes are defined as before, except that now we test against $\varphi \in C^\infty(\T)$ and all noises are replaced by their periodic counterparts. Then it is easy to adapt the proof of Proposition~\ref{prop:kpz drift} to show that also in the periodic setting the Burgers drift $\int_0^\cdot \partial_x u_s^2 \dd s$ is well-defined; alternatively, see~\cite[Lemma~1]{bib:gubinelliJara}. Thus, we define strong stationary solutions respectively energy solutions to the periodic Burgers equation exactly as in the non-periodic setting. We then have the analogous uniqueness result to Theorem~\ref{thm:burgers-uniqueness}:

\begin{theorem}\label{thm:burgers-uniqueness-periodic}
  Let $W$ be a periodic space-time white noise on $(\Omega, \mathcal{F}, (\mathcal{F}_t)_{t \geqslant 0}, \mathbb{P})$ and
  let $\eta$ be a $\mathcal{F}_0$-measurable periodic space white noise. Then the
  strong stationary $u$ solution to
  \begin{equation}\label{eq:thm-SBE-periodic}
     \dd u_t = \Delta u_t \dd t + \lambda \partial_x  u_t^2 \dd t + \sqrt{2} \partial_x \dd W_t, \qquad u_0 = \eta,
  \end{equation}
  is unique up to indistinguishability. Moreover, for $\lambda \neq 0$ we have
  $u = \lambda^{-1} \partial_x \log Z$, where the derivative is taken in the
  distributional sense and $Z$ is the unique solution to the linear multiplicative stochastic
  heat equation
  \begin{equation}\label{eq:thm-SHE-periodic}
    \dd Z_t = \Delta Z_t \dd t + \sqrt{2} \lambda Z_t \dd W_t, \qquad Z_0 = e^{\lambda I \eta},
  \end{equation}
  for $\CF_\T(I \eta)(k) \assign (2\pi i k)^{-1} \CF_\T \eta(k)$, $k \in \Z \setminus\{0\}$, $\CF_\T(I \eta)(0) = 0$, and where $\CF_\T$ denotes the Fourier transform on $\T$. Consequently, any two energy solutions of~(\ref{eq:thm-SBE-periodic}) have the
  same law.
\end{theorem}

We explain in Appendix~\ref{app:periodic} how to modify the arguments for the non-periodic case in order to prove Theorem~\ref{thm:burgers-uniqueness-periodic}.

\section{Additive functionals of controlled processes}\label{sec:additive}

\subsection{It\^o trick and Kipnis-Varadhan inequality}\label{sec:ito}

Our main method for controlling additive functionals of controlled processes is to write them as a sum of a forward- and a backward-martingale which enables us to apply martingale inequalities. For that purpose we first introduce some notation. Throughout this section we fix $(u, \mathcal{A}) \in \CQ$.
\begin{definition}
   The space of \emph{cylinder functions} $\CC$ consists of all $F\colon \CS'(\R) \to \R$ of the form $F(u) = f(u(\varphi_1), \ldots, u (\varphi_n))$ for $\varphi_1, \ldots, \varphi_n \in \CS(\R)$ and $f \in C^2 (\mathbb{R}^n)$ with polynomial growth of its partial derivatives up to order $2$. 
\end{definition}
For $F \in \CC$ we define the action of the \emph{Ornstein--Uhlenbeck generator} $L_0$ as
\[
   L_0 F(u) \assign \sum_{i=1}^n \partial_i f(u(\varphi_1), \ldots, u (\varphi_n)) u( \Delta \varphi_i) + \sum_{i,j=1}^n \partial_{ij}^2 f(u(\varphi_1), \ldots, u (\varphi_n)) \langle \partial_x \varphi_i, \partial_x \varphi_j\rangle_{L^2(\R)}.
\]
With the help of It\^o's formula it is easy to verify that if $(X,0) \in \CQ$ is the Ornstein--Uhlenbeck process and $F \in \CC$, then $F(X_t) - F(X_0) - \int_0^t L_0 F(X_s) \dd s$, $t \geqslant 0$, is a martingale and in particular $L_0 F$ is indeed the action of the generator of $X$ on $F$. We will see in Corollary~\ref{cor:OU-generator} below that $L_0$ can be uniquely extended from $\CC$ to a closed unbounded operator on $L^2(\mu)$, also denoted by $L_0$, so $\CC$ is a core for $L_0$. We also define the Malliavin derivative
\[
   \dD_x F(u) \assign  \sum_{i=1}^n \partial_i f(u(\varphi_1), \ldots, u (\varphi_n)) \varphi_i(x),\qquad x \in \R,
\]
for all $F \in \CC$, and since $\mu$ is the law of the white noise we are in a standard Gaussian setting and $\dD$ is closable as an unbounded operator from $L^p(\mu)$ to $L^p(\mu; L^2(\R))$ for all $p \in [1, \infty)$, see for example~\cite{bib:nualart}. Similarly also $F \mapsto \partial_x \dD_x F$ is closable from $L^p(\mu)$ to $L^p(\mu; L^2(\R))$ for all $p \in [1, \infty)$, and we denote the domain of the resulting operator by $\CW^{1,p}$. Then $\CW^{1,p}$ is the completion of $\CC$ with respect to the norm $\E[|F|^p]^{1/p} + \E[\| \partial_x \dD F\|_{L^2(\R)}^p]^{1/p}$. So writing
\[
   \mathcal{E}(F(u)) \assign 2 \| \partial_x \dD_x F(u) \|_{L^2(\R)}^2
\]
we have $\mathcal{E}(F(\cdot)) \in L^{p/2}(\mu)$ for all $F \in \CW^{1,p}$. Finally, we denote
\[
   \| F \|_1^2 \assign \E[\mathcal{E}(F(u_0))].
\]
The following martingale or It\^o trick is well known for Markov processes, see for example the monograph~\cite{bib:komorowskiLandimOlla}, and in the case of controlled processes on $\R_+ \times \T$ it is due to~\cite{bib:gubinelliJara}. The proof is in all cases essentially the same.

\begin{proposition}[It{\^o} trick]\label{prop:ito-trick}
  Let $T > 0$, $p \geqslant 1$ and $F \in L^p([0,T]; \CW^{1,p})$. Then we have for all $(u, \mathcal{A}) \in \CQ$
  \begin{equation}\label{eq:ito-trick}
     \mathbb{E} \left[ \sup_{t\leqslant T} \left| \int_0^t L_0 F (s, u_s) \dd s \right|^p \right] \lesssim T^{p / 2 - 1} \int_0^T \mathbb{E}  [\mathcal{E} (F (s, u_0))^{p / 2}] \dd s.
  \end{equation}
  For $p = 2$ we get in particular
  \[ \mathbb{E} \left[ \sup_{t \leqslant T} \left| \int_0^t L_0 F (s, u_s)
     \dd s \right|^2 \right] \lesssim  \int_0^T \| F (s, \cdummy) \|_1^2 \dd s. \]
\end{proposition}

\begin{proof}
  We first assume that $F(t) \in \CC$ for all $t \in [0,T]$ and that $t \mapsto F(t,u) \in C^1(\R)$ for all $u \in \CS'(\R)$. Since $(u_t(\varphi))_{t\geqslant 0}$ is a Dirichlet process for all $\varphi \in \CS(\R)$ (the sum of a local martingale and a zero quadratic variation process) we can then apply the It\^o formula for Dirichlet processes, see~\cite{bib:russoVallois}, to $F$ and obtain
  \[
     F (t, u_t) = F (0, u_0) + \int_0^t (\partial_s F (s, u_s) + L_0 F (s, u_s)) \dd s + \int_0^t \partial_u F (s, u_s) \dd \mathcal{A}_s +  M^F_t, \qquad t \geqslant 0,
  \]
  for a continuous martingale $M^F$ with $M^F_0=0$ and quadratic variation $\dd [M^F]_s
  =\mathcal{E} (F (s, u_s)) \dd s$. Similarly we get for $\hat{u}_t = u_{T-t}$
  \begin{align*}
    F (T - T, \hat{u}_T) & = F (T - (T - t), \hat{u}_{T - t}) + \int_{T - t}^T
    (\partial_s F (T - s, \hat{u}_s) + L_0 F (T - s, \hat{u}_s)) \dd s\\
    &\quad + \int_{T - t}^T \partial_u F (T - s, \hat{u}_s) \dd
    \hat{\mathcal{A}}_s + \hat{M}^F_T - \hat{M}^F_{T - t},
  \end{align*}
  for a continuous backward martingale $\hat{M}^F$ with $\hat{M}^F_0 = 0$ and quadratic
  variation $\dd [\hat{M}^F]_s =\mathcal{E} (F (T - s, \hat{u}_s)) \dd s$. Adding these two formulas, we get
 \begin{equation*}
    0 = 2 \int_0^t L_0 F (s, u_s) \dd s + M^F_t + \hat{M}^F_T - \hat{M}^F_{T - t},
  \end{equation*}
 and thus the Burkholder-Davis-Gundy inequality yields
  \begin{align*}
    \mathbb{E} \left[ \sup_{t \leqslant T} \left| \int_0^t L_0 F (s, u_s) \dd s \right|^p \right] & \leqslant \mathbb{E} \Big[\sup_{t \leqslant T} | M^F_t + \hat{M}^F_T - \hat{M}^F_{T - t} |^p\Big] \lesssim \mathbb{E} [[M^F]_T^{p / 2}] +\mathbb{E} [ [\hat{M}^F]_T^{p / 2}]\\
    & \simeq \mathbb{E} \left[ \left( \int_0^T \mathcal{E} (F (s, u_s)) \dd  s \right)^{p / 2} \right] \leqslant T^{p / 2 - 1} \int_0^T \mathbb{E} [\mathcal{E} (F (s, u_0))^{p / 2}] \dd s.
  \end{align*}
  For a general $F \in L^p([0,T]; \CW^{1,p})$ we first approximate $F$ in $L^p([0,T]; \CW^{1,p})$ by a step function that is piecewise constant in time, then we approximate each of the finitely many values that the step function takes by a cylinder function, and finally we mollify the jumps of the new step function. In that way our bound extends to all of $L^p([0,T]; \CW^{1,p})$.
\end{proof}

\begin{remark}
   The right hand side of~\eqref{eq:ito-trick} does not involve the $L^p([0,T]; L^p(\mu))$ norm of $F$ and indeed it is possible to extend the result to the following space. Identify all $F, \tilde F \in \CC$ with $\E[\mathcal{E}(F(u_0) - \tilde{F}(u_0))^{p/2}]^{1/p}=0$ and write $\dot{\CW}^{1,p}$ for the completion of the resulting equivalence classes with respect to the norm $F \mapsto \E[\mathcal{E}(F(u_0))^{p/2}]^{1/p}$. Then~\eqref{eq:ito-trick} holds for all $F \in L^p([0,T]; \dot{\CW}^{1,p})$ provided that the integral on the left hand side in~\eqref{eq:ito-trick} makes sense. But we will not need this.
\end{remark}

\begin{remark}\label{rmk:ito with finite chaos}
   If in the setting of Proposition~\ref{prop:ito-trick} $F(s)$ has a finite chaos expansion of length $n$ for all $s \in [0,T]$ (see Section~\ref{sec:gaussian} for the definition), then also $\mathcal{E}(F(s))$ has a chaos expansion of length $n$ and therefore Gaussian hypercontractivity shows that for all $p \geqslant 1$
  \[
     \mathbb{E} \left[ \sup_{t \in [0, T]} \left| \int_0^t L_0 F (s, u_s) \dd s \right|^p \right] \lesssim T^{p / 2 - 1} \int_0^T \| F (s,    \cdummy) \|_1^p \dd s.
  \]
\end{remark}

The bound in Proposition~\ref{prop:ito-trick} is very powerful and allows us to control $\int_0^\cdot F(s,u_s) \dd s$ provided that we are able to solve the Poisson equation
\[
   L_0 G(s) = F(s)
\]
for all $s \in [0,T]$. Note that this is an infinite dimensional PDE which a priori is difficult to solve, but that we only need to consider it in $L^2(\mu)$ which has a lot of structure as a Gaussian Hilbert space. We will discuss this further in Section~\ref{sec:gaussian} below. Nonetheless, we will encounter situations where we are not able to solve the Poisson equation explicitly, and in that case we rely on the method of Kipnis and Varadhan allowing us to bound $\int_0^\cdot F(s,u_s) \dd s$ in terms of a certain variational norm of $F$. We define for $F \in L^2 (\mu)$
\[
  \| F \|_{- 1}^2 \assign \sup_{G \in \CC} \{ 2\mathbb{E} [F(u_0) G(u_0)] - \| G \|_1^2 \},
\]
and we write $F \in \dot{\CH}^{-1}$ if the right hand side is finite. We will need a slightly refined version of the Kipnis-Varadhan inequality which controls also the $p$-variation. Recall that
for $p \geqslant 1$ the $p$-variation of $f \colon [0, T] \rightarrow \mathbb{R}$
is
\[ \| f \|_{p - \tmop{var} ; [0, T]} \assign \sup \left\{ \left( \sum_{k =
   0}^{n - 1} | f (t_{k + 1}) - f (t_k) |^p \right)^{1 / p} : n \in
   \mathbb{N}, 0 = t_0 < \ldots < t_n = T \right\} . \]

\begin{corollary}[Kipnis-Varadhan inequality]\label{cor:KV}
  Let $T > 0$ and $F \in L^2 ( [0, T], \dot{\CH}^{- 1} \cap L^2 (\mu) )$,
  and let $(u, \mathcal{A}) \in \CQ$ be a controlled process. Then for all $p > 2$
  \[ \mathbb{E} \left[ \left\| \int_0^{\cdummy} F (s, u_s) \dd s
     \right\|_{p - \tmop{var} ; [0, T]}^2 \right] +\mathbb{E} \left[ \sup_{t
     \leqslant T} \left| \int_0^t F (s, u_s) \dd s \right|^2 \right]
     \lesssim \int_0^T \| F (s, \cdummy) \|_{- 1}^2 \dd s. \]
\end{corollary}

\begin{proof}
  The non-reversible version of the Kipnis-Varadhan inequality is due to~\cite{bib:wu}, and our proof is essentially the same as in~\cite{bib:komorowskiLandimOlla, bib:funakiQuastel}. But we are not aware of any reference for the statement about the $p$-variation. Note that since the integral vanishes in zero, its supremum norm can be controlled by its $p$-variation. Let $H \in L^2 ( [0, T] ; \CW^{1,2} )$ and
  decompose
  \begin{equation}
    \label{eq:KV pr1} \int_0^{\cdummy} F (s, u_s) \dd s = \int_0^{\cdummy}
    L_0 H (s, u_s) \dd s + \int_0^{\cdummy} (F - L_0 H) (s, u_s) \dd s.
  \end{equation}
  For the first term on the right hand side we apply the same martingale
  decomposition as in the proof of the It{\^o} trick to get $\int_0^t L_0 H
  (s, u_s) \dd s = - 1 / 2 (M^H_t + \hat{M}^H_T - \hat{M}^H_{T - t})$. By~\cite[Proposition~2]{bib:lepingle} we can therefore control the $p$-variation by
  \[ \mathbb{E} \left[ \left\| \int_0^{\cdummy} L_0 H (s, u_s) \dd s
     \right\|_{p - \tmop{var} ; [0, T]}^2 \right] \lesssim \mathbb{E} \left[
     \sup_{t \leqslant T} \left| \int_0^t L_0 H (s, u_s) \dd s \right|^2
     \right] \lesssim \int_0^T \| H (s, \cdummy) \|_1^2 \dd s, \]
  where the second inequality follows from Proposition~\ref{prop:ito-trick}. For the second
  term on the right hand side of~(\ref{eq:KV pr1}) we get
  \begin{align*}
    \mathbb{E} \left[ \left\| \int_0^{\cdummy} (F - L_0 H) (s, u_s) \dd s
    \right\|_{p - \tmop{var} ; [0, T]}^2 \right] & \leqslant \mathbb{E} \left[
    \left\| \int_0^{\cdummy} (F - L_0 H) (s, u_s) \dd s \right\|_{1 -
    \tmop{var} ; [0, T]}^2 \right]\\
    & \leqslant \mathbb{E} \left[ \left( \int_0^T | (F - L_0 H) (s, u_s) |  \dd s \right)^2 \right] \\
    & \leqslant T \int_0^T \| (F - L_0 H) (s) \|_{L^2
    (\mu)}^2 \dd s,
  \end{align*}
  and therefore overall
  \[
     \mathbb{E} \left[ \left\| \int_0^{\cdummy} F (s, u_s) \dd s  \right\|_{p - \tmop{var} ; [0, T]}^2 \right] \lesssim \int_0^T (\| H (s,     \cdummy) \|_1^2 + T \| (F - L_0 H) (s) \|_{L^2 (\mu)}^2) \dd s.
  \]
  Now take $H_{\lambda} (s)$ as the solution to the resolvent equation
  $(\lambda - L_0) H_{\lambda} (s) = - F (s)$. Note that unlike the Poisson
  equation, the resolvent equation is always solvable and $H_{\lambda} (s) = -
  \int_0^{\infty} e^{- \lambda r} P^{\tmop{OU}}_r F (s) \dd r$, where
  $P^{\tmop{OU}}$ is the semigroup generated by $L_0$. Then $(F - L_0
  H_{\lambda}) (s) = - \lambda H_{\lambda} (s)$ and by Lemma~\ref{lem:H1 norm} below we have $\|H_\lambda(s) \|_1^2 = 2 \langle H_\lambda(s), (-L_0) H_\lambda(s) \rangle_{L^2(\mu)}$, which yields
  \begin{gather*}
    \lambda \| H_{\lambda} (s) \|_{L^2 (\mu)}^2 + \frac{1}{2} \| H_{\lambda} (s) \|_1^2 =
    \langle H_{\lambda} (s), (\lambda - L_0) H_{\lambda} (s) \rangle_{L^2
    (\mu)}\\
    = \langle H_{\lambda} (s), - F (s) \rangle_{L^2 (\mu)} \leqslant \|
    H_{\lambda} (s) \|_1 \| F (s) \|_{- 1},
  \end{gather*}
  from where we get $\| H_{\lambda} (s) \|_1 \leqslant 2 \| F (s) \|_{- 1}$ and
  then also $\| H_{\lambda} (s) \|_{L^2 (\mu)}^2 \leqslant \lambda^{- 1} 2 \| F
  (s) \|^2_{- 1}$. Therefore,
  \[ \| H_{\lambda} (s) \|_1^2 + T \| (F - L_0 H_{\lambda}) (s)
     \|_{L^2 (\mu)}^2 = \| H_{\lambda} (s) \|_1^2 + T \| \lambda
     H_{\lambda} (s) \|_{L^2 (\mu)}^2 \lesssim \| F (s) \|_{- 1}^2 + \lambda
     T \| F (s) \|_{- 1}^2, \]
  and now it suffices to send $\lambda \rightarrow 0$.
\end{proof}

\subsection{Gaussian analysis}\label{sec:gaussian}

To turn the It\^o trick or the Kipnis-Varadhan inequality into a useful bound we must be able to either solve the Poisson equation $L_0 G = F$ for a given $F$ or to control the variational norm appearing in the Kipnis-Varadhan inequality. Here we discuss how to exploit the Gaussian structure of $L^2(\mu)$ in order to do so. For details on Gaussian Hilbert spaces we refer to~\cite{bib:janson, bib:nualart}. Since $L^2 (\mu)$ is a Gaussian Hilbert space, we have the orthogonal decomposition
\[
   L^2(\mu) = \bigoplus_{n \geqslant 0} \mathcal{H}_n,
\]
where $\mathcal{H}_n$ is the closure in $L^2(\mu)$ of the span of all random variables of the form $u\mapsto H_n(u(\varphi))$, with $H_n (x) = e^{x^2 / 2} (- 1)^n \partial^n_x e^{- x^2 / 2}$ being the $n$-th Hermite polynomial and where $\varphi \in \CS(\R)$ with $\| \varphi\|_{L^2(\R)} = 1$. The space $\mathcal{H}_n$ is called the \emph{$n$-th homogeneous chaos}, and $\bigoplus_{k=0}^n \mathcal{H}_k$ is the \emph{$n$-th inhomogeneous chaos}. Also, $\mu$ is the law of the white noise on $\R$ and therefore we can identify
\[
   \mathcal{H}_n = \{ W_n(f_n) : f_n \in L^2_s(\R^n) \},
\]
where $W_n(f_n)$ is the multiple Wiener-It\^o integral of $f_n \in L^2_s (\mathbb{R}^n)$, that is
\[
   W_n (f_n) \assign \int_{\mathbb{R}^n} f (y_1, \ldots, y_n) W (\dd y_1 \ldots \dd y_n).
\]
Here $L^2_s (\mathbb{R}^n)$ are the equivalence classes of $L^2 (\mathbb{R}^n)$ that are induced by the seminorm
\[ \| f \|_{L^2_s (\mathbb{R}^n)} \assign \| \tilde{f} \|_{L^2_s
   (\mathbb{R}^n)}, \qquad \tilde{f} (x_1, \ldots, x_n) = \frac{1}{n!}
   \sum_{\sigma \in \mathcal{S}_n} f (x_{\sigma (1)}, \ldots, x_{\sigma (n)}),
\]
where $\mathcal{S}_n$ denotes the set of permutations of $\{1,\dots,n\}$. Of course, $\lVert \cdot \rVert_{L^2_s (\mathbb{R}^n)}$ is a norm on $L^2_s (\mathbb{R}^n)$ and we usually identify an equivalence class in $L^2_s (\mathbb{R}^n)$ with its symmetric representative. The link between the multiple stochastic integrals $W_n$ and the Malliavin derivative $\dD$ is explained in the following partial integration by parts rule which will be used for some explicit computations below.

\begin{lemma}
  \label{lem:IBP}Let $f \in L^2_s (\mathbb{R}^n)$ and let $F$ be Malliavin differentiable in $L^2(\mu)$. Then
  \[ \mathbb{E} [W_n (f) F] = \int_{\mathbb{R}} \mathbb{E} [W_{n - 1}
     (f (y, \cdummy)) \dD_y F] \dd y \]
\end{lemma}

\begin{proof}
  The proof is virtually the same as for~\cite[Lemma~1.2.1]{bib:nualart}. Since the span of functions of the form $f = f_1
  \otimes \ldots \otimes f_n$ is dense in $L^2_s (\mathbb{R}^n)$, it suffices
  to argue for such $f$. By polarization it suffices to consider $f_1 = \ldots
  = f_n$ with $\| f_1 \|_{L^2 (\mathbb{R})} = 1$, for which $W_n (f) = H_n
  (W_1 (f_1))$ for the $n$-th Hermite polynomial $H_n$. By another
  approximation argument we may suppose that $F = \Phi (W_1 (f_1), W_1
  (\varphi_1), \ldots, W_1 (\varphi_m))$ for orthonormal $\varphi_1, \ldots, \varphi_m \in L^2
  (\mathbb{R})$ that are also orthogonal to $f_1$ and for $\Phi \in
  C^{\infty}_c(\R^{m+1})$. So if $\nu_{m + 1}$ denotes the $(m + 1)$-dimensional
  standard normal distribution, then
  \begin{align*}
    \mathbb{E} [W_n (f) F] & =\mathbb{E} [H_n (W_1 (f_1)) \Phi (W_1 (f_1),
    W_1 (\varphi_1), \ldots, W_1 (\varphi_m))]\\
    &= \int_{\mathbb{R}^{m + 1}} H_n (x_1) \Phi (x_1, x_2, \ldots, x_{m + 1})
    \nu_{m + 1} (\dd x_1, \ldots, \dd x_{m + 1})\\
    &= \int_{\mathbb{R}^{m + 1}} H_{n - 1} (x_1) \partial_{x_1} \Phi (x_1,
    x_2, \ldots, x_{m + 1}) \nu_{m + 1} (\dd x_1, \ldots, \dd x_{m +
    1})\\
    &=\mathbb{E} \left[ H_{n - 1} (W_1 (f_1)) \int_{\mathbb{R}} \dD_y F
    f_1 (y) \dd y \right] = \int_{\mathbb{R}} \mathbb{E} [W_{n - 1} (f
    (y, \cdummy)) \dD_y F] \dd y,
  \end{align*}
  which concludes the proof.
\end{proof}

Recall that so far we defined the operator $L_0$ acting on cylinder functions. If we consider a cylinder function $F \in \mathcal{H}_n$ for some given $n$, then the action of $L_0$ is particularly simple.

\begin{lemma}\label{lem:generator-chaos}
   Let $n \geqslant 0$ and $F \in \CC$ be such that in $L^2(\mu)$ we have $F = W_n(f_n)$ for $f_n \in H^2_s (\mathbb{R}^n) \subset L^2_s (\mathbb{R}^n)$, the twice weakly differentiable symmetric functions from $\mathbb{R}^n$ to $\mathbb{R}$ that satisfy
  \[ \| f_n \|^2_{H^2_s (\mathbb{R}^n)} \assign \| f_n \|^2_{L^2_s
     (\mathbb{R}^n)} + \| \Delta f_n \|^2_{L^2_s (\mathbb{R}^n)} < \infty .
  \]
  Then
  \[
     L_0 F =  L_0 W(f_n) = W_n (\Delta f_n) \quad \text{in } L^2(\mu).
  \]
\end{lemma}

\begin{proof}
  Consider first a functional $G \in \CC$ of the form $G (u) = H_n (u (\varphi))$, where $\varphi \in \CS(\R)$ with $\| \varphi \|_{L^2 (\mathbb{R})} = 1$. In that case
  \begin{align}\label{eq:generator-chaos pr1}
    L_0 F (u) & = H_n' (u (\varphi)) u (\Delta \varphi) + H_n'' (u(\varphi))  \langle \partial_x \varphi, \partial_x \varphi \rangle_{L^2(\R)} \\
    & = n H_{n - 1} (u(\varphi)) u (\Delta \varphi) - n (n - 1) H_{n - 2} (u (\varphi)) \langle \varphi,  \Delta \varphi \rangle_{L^2(\R)},
  \end{align}
  where in the second step we used that $H_k' = k H_{k - 1}$ for $k \geqslant 1$ and $H'_0 = 0$. Now we use that $\| \varphi
  \|_{L^2 (\mathbb{R})} = 1$ to rewrite $H_k (W_1 (\varphi)) = W_k (\varphi^{\otimes k})$, see Proposition~1.1.4 of~\cite{bib:nualart} and note the additional factor $k!$ in our definition of $H_k$ compared to the
  one in~\cite{bib:nualart}. Thus, we can apply Proposition~1.1.2 of~\cite{bib:nualart} to compute the first term on the right hand side:
  \begin{align*}
    n H_{n - 1} (X (\varphi)) X (\Delta \varphi) & = n W_{n - 1} (\varphi^{\otimes (n - 1)}) W_1 (\Delta \varphi)\\
    & = n W_n (\varphi^{\otimes (n - 1)} \otimes \Delta \varphi) + n (n - 1) W_{n - 2} (\varphi^{\otimes (n - 2)}) \langle \varphi, \Delta \varphi \rangle_{L^2(\R)}\\
     & = W_n (\Delta \varphi^{\otimes n}) + n (n - 1) H_{n - 2} (X (\varphi)) \langle \varphi, \Delta \varphi \rangle_{L^2(\R)},
  \end{align*}
  where in the first term on the right hand side $\Delta$ denotes the Laplacian on $\mathbb{R}^n$. Plugging this back into~(\ref{eq:generator-chaos pr1}) we obtain $L_0 W_n (\varphi^{\otimes n}) = W_n (\Delta
  \varphi^{\otimes n})$. By polarization this extends to $W_n (\varphi_1
  \otimes \ldots \otimes \varphi_n)$, and then to general $\varphi \in H^2_s
  (\mathbb{R}^n)$ by taking the closure of the span of functions of the form
  $\varphi_1 \otimes \ldots \otimes \varphi_n$ with $\varphi_i \in \CS(\R)$.
\end{proof}

\begin{corollary}\label{cor:OU-generator}
   The operator $L_0$ is closable in $L^2(\mu)$ and the domain of its closure, still denoted with $L_0$, is
  \[ \tmop{dom} (L_0) = \left\{ F = \sum_{n \geqslant 0} W_n (f_n) : f_n \in
     H^2_s (\mathbb{R}^n) \text{ for all } n \in \mathbb{N} \text{ and }
     \sum_{n \geqslant 0} n! \| f_n \|_{H^2_s (\mathbb{R}^n)}^2 < \infty
     \right\} . \]
  For $F \in \tmop{dom} (L_0)$ we have
  \begin{equation}
    \label{eq:L0-def} L_0 F = \sum_{n \geqslant 0} W_n (\Delta f_n) .
  \end{equation}
\end{corollary}

\begin{proof}
  Let $F(u) = f(u(\varphi_1), \dots, u(\varphi_n)) \in \CC$ with chaos expansion $F = \sum_{n \geqslant 0} W_n( f_n)$. By a standard approximation argument it follows that $L_0 F = \sum_{n \geqslant 0} W_n(\Delta f_n)$. But then formula (iii) on page~9 of~\cite{bib:nualart} yields
  \[
     \E[|F|^2] = \sum_{n \geqslant 0} n! \| f_n \|_{L^2_s(\R^n)}^2, \qquad \E[| L_0 F|^2] = \sum_{n \geqslant 0} n! \| \Delta f_n \|_{L^2_s(\R^n)}^2,
  \]
  from where our claim readily follows because the Laplace operator on $L^2_s(\R^n)$ is a closed operator with domain $H^2_s(\R^n)$.
\end{proof}

Before we continue, let us link the $\| \cdot \|_1$ norm defined in Section~\ref{sec:ito} with the operator $L_0$.

\begin{lemma}\label{lem:H1 norm}
  For $F \in \CC$ we have
  \[
     \| F \|_1^2 =\mathbb{E} [\mathcal{E} (F (u_0))] = 2 \E[ F(u_0) (-L_0 F)(u_0)].
  \]
\end{lemma}

\begin{proof}
   See~\cite[Section~2.4]{bib:ebp} for a proof in the periodic case which works also in our setting.
\end{proof}

Next, we define two auxiliary Hilbert spaces which will be useful in controlling additive functionals of controlled processes.

\begin{definition}
  We identify all $F, \tilde F \in \CC$ with $\|F - \tilde{F} \|_1 = 0$, and we write $\dot{\CH}^1$ for the completion of the equivalence classes of $\CC$ with respect to $\lVert \cdot \rVert_1$.
  
  Similarly we identify $F, \tilde F \in L^2(\mu)$ with $\| F \|_{- 1} + \| \tilde F \|_{- 1} < \infty$ if $\|F - \tilde{F} \|_{- 1} = 0$ and the space $\dot{\CH}^{- 1}$ is defined as the completion of the equivalence classes with respect to $\lVert \cdot \rVert_{- 1}$.
\end{definition}

\begin{definition}
  Recall that for $r \in \mathbb{R}$ and $n \in \mathbb{N}$, the homogeneous
  Sobolev space $\dot{H}^r (\mathbb{R}^n)$ is constructed in the following
  way: We set for $f \in \CS(\R)$
  \[ \| f \|_{\dot{H}^r (\mathbb{R}^n)}^2 \assign \int_{\mathbb{R}^n} |
     \hat{f} (z) |^2 | 2 \pi z |^{2 r} \dd z \in [0, \infty], \]
  and consider only those $f$ with $\| f \|_{\dot{H}^r(\R^n)} < \infty$, where we
  identify $f$ and $g$ if $\|f - g \|_{\dot{H}^r(\R^n)} = 0$. The
  space $\dot{H}^r (\mathbb{R}^n)$ is then the completion of the equivalence
  classes with respect to $\lVert \cdot \rVert_{\dot{H}^r(\R^n)}$.
  
  We write $\dot{H}^r_s (\mathbb{R}^n)$ for the space that is obtained if we
  perform the same construction replacing $\lVert \cdot \rVert_{\dot{H}^r
  (\mathbb{R}^n)}$ by
  \[ \| f \|_{\dot{H}_s^r (\mathbb{R}^n)}^2 \assign \int_{\mathbb{R}^n} |
     \widehat{\tilde{f}} (z) |^2 | 2\pi z |^{2 r} \dd z, \qquad \tilde{f}
     (x_1, \ldots, x_n) = \frac{1}{n!} \sum_{\sigma \in \mathcal{S}_n} f
     (x_{\sigma (1)}, \ldots, x_{\sigma (n)}).
  \]
\end{definition}

\begin{remark}
  By construction, $\dot{H}^r_s (\mathbb{R}^n)$ is always a Hilbert space.
  For $r < n / 2$ there is an explicit description, see~\cite{bib:bahouri}, Propositions~1.34 and~1.35 for the non-symmetric
  case:
  \[ \dot{H}^r_s (\mathbb{R}^n) = \left\{ f \in \CS'_s(\R^n) : \hat{f} \in
     L^1_{\tmop{loc}} \text{ and } \int_{\mathbb{R}^n} | \hat{f} (z) |^2 | 2\pi z
     |^{2 r} \dd z < \infty \right\} . \]
  Here we wrote $\CS'_s(\R^n)$ for those tempered distributions $u$ with $u
  (\varphi) = u (\tilde{\varphi})$ for all $\varphi \in \CS(\R^n)$, where $\tilde \varphi$ is the symmetrization of $\varphi$.
\end{remark}

\begin{lemma}\label{lem:H1 chaos}
  For $f_n \in \CS_s (\mathbb{R}^n)$ and $n \geqslant 1$ we have
  \[ \| W_n (f_n) \|_1^2 = n! \| f_n \|_{\dot{H}^1_s (\mathbb{R}^n)}^2,
     \qquad \| W_n (f_n) \|_{- 1}^2 = n! \| f_n \|_{\dot{H}^{- 1}_s
     (\mathbb{R}^n)}^2 . \]
\end{lemma}

\begin{proof}
  For the $\dot{\CH}^1$ norm it suffices to note that
  \[ \mathbb{E}_{\mu} [W_n (f_n) (- L_0) W_n (f_n)] = n! \| \nabla f_n \|_{L^2_s(\R^n)}^2 = n! \| f_n \|_{\dot{H}^1_s
     (\mathbb{R}^n)}^2, \]
  where the last equality follow from Plancherel's formula. For the $\dot{\CH}^{-
  1}$ norm let us consider a test function $G = \sum_{m \geqslant 0} W_m (g_m) \in \CC$.
  Then
  \begin{align*}
     2\mathbb{E} [W_n (f_n)(u_0) G(u_0)] - \| G \|_1^2 & = 2 n! \langle f_n, g_n \rangle_{L^2_s (\mathbb{R}^n)} - \sum_{m \geqslant 0} m! \| g_m \|_{\dot{H}^1_s (\mathbb{R}^m)}^2 \\
     & \leqslant 2 n! \langle f_n, g_n \rangle_{L^2_s (\mathbb{R}^n)} - n! \| g_n \|_{\dot{H}^1_s (\mathbb{R}^n)}^2,
  \end{align*}
  and Plancherel's formula and then the Cauchy-Schwarz inequality give
  \begin{align*}
    \langle f_n, g_n \rangle_{L^2_s (\mathbb{R}^n)} = \int_{\mathbb{R}^n} \widehat{f_n} (z) (\hat{g}_n (z))^{\ast} \dd z & \leqslant \left( \int_{\mathbb{R}^n} | \widehat{f_n} (z) |^2 | 2 \pi z |^{- 2} \dd z \right)^{1 / 2} \left( \int_{\mathbb{R}^n} | \widehat{g_n} (z) |^2 | 2 \pi z    |^2 \dd z \right)^{1 / 2}\\
    & = \| f_n \|_{\dot{H}^{- 1}_s (\mathbb{R}^n)} \| g_n \|_{\dot{H}^1_s (\mathbb{R}^n)},
  \end{align*}
  from where we get $\| W_n (f_n) \|_{- 1}^2 \leqslant n! \| f_n \|_{\dot{H}^{- 1}_s (\mathbb{R}^n)}^2$. To see the converse inequality let $f_n \in \dot{H}^{- 1}_s
  (\mathbb{R}^n)$. Then we have for $\hat{g} (z) \assign (\widehat{f_n}
  (z))^{\ast} | 2 \pi z |^{- 2}$
  \[ 2 \int_{\mathbb{R}^n} \widehat{f_n} (z) \hat{g} (z) \dd z -
     \int_{\mathbb{R}^n} | \hat{g} (z) |^2 | 2 \pi z |^2 \dd z = 2 \| f_n
     \|_{\dot{H}^{- 1}_s}^2 - \int_{\mathbb{R}^n} |
     (\widehat{f_n} (z))^{\ast} | 2 \pi z |^{- 2} |^2 | 2 \pi z |^2 \dd z = \| f_n
     \|_{\dot{H}^{- 1}_s}^2 . \]
  Of course $W_n (g)$ may not be in $\CC$, but we can approximate it by
  functions in $\CC$ and this concludes the proof.
\end{proof}

\begin{corollary}
  \label{cor:L0-extension}We have an explicit representation of $\dot{\CH}^1$ and
  $\dot{\CH}^{- 1}$ via
  \[ \dot{\CH}^1 = \left\{ \sum_{n \geqslant 1} W_n (f_n) : f_n \in \dot{H}^1_s
     (\mathbb{R}^n), \ \sum_n n! \| f_n \|_{\dot{H}_s^1 (\mathbb{R}^n)}^2 <
     \infty \right\}, \]
  \[ \dot{\CH}^{- 1} = \left\{ \sum_{n \geqslant 1} W_n (f_n) : f_n \in \dot{H}^{-
     1}_s (\mathbb{R}^n), \ \sum_n n! \| f_n \|_{\dot{H}_s^{- 1}
     (\mathbb{R}^n)}^2 < \infty \right\} . \]
  Moreover, there is a unique extension of $L_0$ from $\tmop{dom}
  (L_0) \subset L^2 (\mu)$ to $\dot{\CH}^1$ for which $L_0$ is an isometry from
  $\dot{\CH}^1$ to $\dot{\CH}^{- 1}$.
\end{corollary}

\begin{proof}
  We only have to prove the statement about the extension of $L_0$. Since $1
  + |2 \pi z |^4 > | 2 \pi z |^2$ for all $z \in \mathbb{R}^n$ we have $H^2_s (\mathbb{R}^n) \subset \dot{H}_s^1
  (\mathbb{R}^n)$ and $\| f_n \|_{\dot{H}_s^1 (\mathbb{R}^n)}^2 \leqslant \| f_n \|_{H^2_s (\mathbb{R}^n)}^2$ for all $f_n \in H^2_s (\mathbb{R}^n)$. Thus,
  \[ \tmop{dom}_0 (L_0) \assign \{ F \in \tmop{dom} (L_0) : \mathbb{E} [F] =
     0 \} \subset \dot{\CH}^1 . \]
  For $F = \sum_{n \geqslant 1} W_n (f_n) \in \tmop{dom}_0 (L_0)$ we have
  \[ \| L_0 F \|_{- 1}^2 = \sum_{n \geqslant 1} n! \| \Delta f_n
     \|_{\dot{H}_s^{- 1} (\mathbb{R}^n)}^2, \]
  and
  \[ \| \Delta f_n \|_{\dot{H}_s^{- 1} (\mathbb{R}^n)}^2 =
     \int_{\mathbb{R}^n} (|2 \pi z |^2 | \widehat{f_n} (z) |)^2 | 2 \pi z |^{- 2} \dd
     z = \| f_n \|_{H^1_s (\mathbb{R}^n)}^2, \]
  which proves that $\| L_0 F \|_{- 1}^2 = \| F \|^2_1$ and thus that $L_0$ is
  an isometry from $\tmop{dom}_0 (L_0) \subset \dot{\CH}^1$ to $\dot{\CH}^{- 1}$ which can
  be uniquely extended to all of $\dot{\CH}^1$ because $\tmop{dom}_0 (L_0)$ is dense in $\dot{\CH}^1$.
\end{proof}

\subsection{The Burgers and KPZ nonlinearity}\label{sec:nonlinearity}

With the tools we have at hand, it is now straightforward to construct the KPZ nonlinearity (and in particular the Burgers nonlinearity) for all controlled processes.

\begin{proposition}\label{prop:kpz drift}
  Let $T>0$, $p \geqslant 1$, let $\rho, \varphi \in \CS(\R)$, and let $(u, \mathcal{A}) \in \CQ$. Then
  \begin{equation}
    \label{eq:burgers-nonlin-1} \mathbb{E} \left[ \sup_{t \leqslant T} \left|
    \int_0^t [(u_s \ast \rho)^2 - \| \rho \|_{L^2(\R)}^2] (\varphi) \dd s
    \right|^p \right] \lesssim (T^{p / 2} \vee T^p) (\| \rho \|_{L^1(\R)}^2 \|
    \varphi \|_{H^{- 1 / 2}(\R)})^p .
  \end{equation}
  Therefore, the integral $\int_0^{\cdummy} [(u_s \ast \rho)^2 - \| \rho
  \|_{L^2(\R)}] (\varphi) \dd s$ is well defined also for $\rho \in L^1(\R) \cap L^2(\R)$. If
  $(\rho^N) \subset L^1(\R) \cap L^2(\R)$ is an approximate identity (i.e. $\sup_N \| \rho^N
  \|_{L^1(\R)} < \infty$, $\widehat{\rho^N} (0) = 1$ for all $N$, $\lim_{N \rightarrow
  \infty} \widehat{\rho^N} (x) = 1$ for all $x$), then there exists a process
  $\int_0^{\cdummy} u_s^{\diamond 2} \dd s \in C \left( \R_+, \CS'(\R)
  \right)$, given by
  \[ \left( \int_0^t u_s^{\diamond 2} \dd s \right) (\varphi) \assign
     \lim_{n \rightarrow \infty} \int_0^t [(u_s \ast \rho^N)^2 - \| \rho^N
     \|_{L^2(\R)}^2] (\varphi) \dd s, \qquad t \geqslant 0, \varphi \in \CS(\R), \]
  where the convergence is in $L^p(\P)$ uniformly on compacts and as the notation suggests the limit does not depend on the approximate
  identity $(\rho^N)$. Furthermore, if $\rho \in L^1(\R)$ is such that $x \rho \in L^1(\R)$ and $\rho^N := N \rho(N \cdot)$, then
  \begin{align}\label{eq:KPZ-nonlin-bound} \nonumber
    &\mathbb{E} \left[ \sup_{t \leqslant T} \left| \int_0^t u_s^{\diamond 2}
    \dd s (\varphi) - \int_0^t [(u_s \ast \rho^N)^2 - \| \rho^N \|_{L^2(\R)}^2]
    (\varphi) \dd s \right|^p \right] \\
    &\hspace{100pt}\lesssim (T^{p / 2} \vee T^p) N^{- p / 2} \{ (1 + \| \rho \|_{L^1(\R)})^{3 / 2} \| x \rho \|^{1 / 2}_{L^1(\R)} \| \varphi \|_{L^2(\R)} \}^p .
  \end{align}
\end{proposition}

\begin{proof}
  Let us set
  \begin{align*}
    F (u) & \assign (u \ast \rho)^2 (\varphi) - \| \rho \|_{L^2 (\mathbb{R})}^2
    \int \varphi (x) \dd x = \int_{\mathbb{R}} (u (\rho (x - \cdummy)) u
    (\rho (x - \cdummy)) - \| \rho \|_{L^2 (\mathbb{R})}^2) \varphi (x)
    \dd x\\
    &= \int_{\mathbb{R}} W_2 (\rho (x - \cdummy)^{\otimes 2})(u) \varphi (x)
    \dd x = W_2 \left( \int_{\mathbb{R}} \rho (x - \cdummy)^{\otimes 2}
    \varphi (x) \dd x \right) (u),
  \end{align*}
  where in the last step we applied the stochastic Fubini theorem. Due to infrared problems it seems
  impossible to directly solve the Poisson equation $L_0 G = F$, so instead we
  consider $H$ with $(1 - L_0) H = F$ which means that $H = W_2 (h)$ for
  \[ (1 - \Delta) h (y_1, y_2) = \int_{\mathbb{R}} \rho (x - y_1) \rho (x -
     y_2) \varphi (x) \dd x, \]
  or in Fourier variables
  \[ \hat{h} (\xi_1, \xi_2) = \frac{\hat{\rho} (- \xi_1) \hat{\rho} (- \xi_2)
     \hat{\varphi} (\xi_1 + \xi_2)}{1 + | 2 \pi \xi |^2} . \]
  Then we have
  \begin{align}\label{eq:kpz drift pr1} \nonumber
    \mathbb{E} \left[ \sup_{t \leqslant T} \left| \int_0^t F (u_s) \dd s
    \right|^p \right] & =\mathbb{E} \left[ \sup_{t \leqslant T} \left| \int_0^t
    (1 - L_0) H (u_s) \dd s \right|^p \right]\\
    & \lesssim \mathbb{E} \left[ \sup_{t \leqslant T} \left| \int_0^t H (u_s)
    \dd s \right|^p \right] +\mathbb{E} \left[ \sup_{t \leqslant T} \left|
    \int_0^t L_0 H (u_s) \dd s \right|^p \right] .
  \end{align}
  For the first term on the right hand side we further apply Gaussian hypercontractivity to estimate
  \begin{align*}
    \mathbb{E} \left[ \sup_{t \leqslant T} \left| \int_0^t H (u_s) \dd s  \right|^p \right] & \leqslant T^p \mathbb{E} [| H (u_0) |^p] \simeq T^p  \mathbb{E} [| H (u_0) |^2]^{p / 2} \simeq T^p \| h \|_{L^2(\R)}^p\\
    & \simeq T^p \left( \int_{\mathbb{R}^2} \frac{| \hat{\rho} (- \xi_1) |^2 | \hat{\rho} (- \xi_2) |^2 | \hat{\varphi} (\xi_1 + \xi_2) |^2}{(1 + |2 \pi \xi  |^2)^2} \dd \xi \right)^{p / 2} \\
    & \leqslant T^p \| \rho \|_{L^1(\R)}^{2 p} \left( \int_{\mathbb{R}^2} \frac{| \hat{\varphi} (\xi_1 + \xi_2) |^2}{(1 + | \xi |^2)^2} \dd \xi \right)^{p / 2} \\
    & \lesssim T^p \| \rho \|_{L^1(\R)}^{2 p} \left( \int_{\mathbb{R}} \frac{| \hat{\varphi} (\xi_1) |^2}{(1 + | \xi_1 |^2)^{3  / 2}} \dd \xi_1 \right)^{p / 2} \simeq T^p \| \rho \|_{L^1(\R)}^{2 p} \| \varphi \|_{H^{- 3 / 2}(\R)}^p,
  \end{align*}
  where we used the completion of the square to compute
  \begin{align*}
    \int_{\mathbb{R}} \frac{1}{(1 + | \xi_1 - \xi_2 |^2 + | \xi_2 |^2)^2} \dd \xi_2 & =  \int_{\mathbb{R}} \frac{1}{(1 + 2 | \xi_2 -  \xi_1 / 2 |^2 + | \xi_1 |^2 / 2)^2} \dd \xi_2 \\
    &= \int_{\mathbb{R}} \frac{1}{(1 + 2 | \xi_2 |^2 + | \xi_1 |^2 / 2)^2} \dd \xi_2 \lesssim \frac{1}{(1 + | \xi_1 |^2)^{3  / 2}}.
  \end{align*}
  By Proposition~\ref{prop:ito-trick} together with Remark~\ref{rmk:ito with finite chaos} and Lemma~\ref{lem:H1 chaos} the second term on the right hand side of~\eqref{eq:kpz drift pr1} is bounded by
  \begin{align*}
    \mathbb{E} \left[ \sup_{t \leqslant T} \left| \int_0^t L_0 H (u_s) \dd
    s \right|^p \right] & \lesssim T^{p / 2} \| h \|_{\dot{H}^1_s (\mathbb{R}^2)}^p = T^{p / 2} \left( \int_{\mathbb{R}^2} \left| \frac{\hat{\rho} (- \xi_1) \hat{\rho} (- \xi_2) \hat{\varphi} (\xi_1 + \xi_2)}{1 + | 2 \pi \xi |^2} \right|^2 |2 \pi \xi |^2 \dd \xi \right)^{p / 2}\\
    & \leqslant T^{p / 2} \| \rho \|_{L^1(\R)}^{2 p} \left( \int_{\mathbb{R}^2} \frac{| \hat{\varphi} (\xi_1 + \xi_2) |^2}{1 + | \xi |^2} \dd \xi \right)^{p / 2} \\
    & \lesssim T^{p / 2} \| \rho \|_{L^1(\R)}^{2 p} \left( \int_{\mathbb{R}} \frac{| \hat{\varphi} (\xi_1) |^2}{(1 + | \xi_1 |^2)^{1  / 2}} \dd \xi_1 \right)^{p / 2} \simeq T^{p / 2} \| \rho \|_{L^1(\R)}^{2 p} \| \varphi \|_{H^{- 1 / 2}(\R)}^p .
  \end{align*}
  Since $\| \varphi \|_{H^{- 1 / 2}} \geqslant \| \varphi \|_{H^{- 3 / 2}}$, the claimed estimate~(\ref{eq:burgers-nonlin-1}) follows.
  
  If $(\rho^N)$ is an approximate identity, then $\widehat{\rho^N} (\xi_1)
  \widehat{\rho^N} (\xi_2)$ converges to $1$ for all $\xi_1, \xi_2$, and since $\|\rho^N \|_{L^1}$ and thus $\| \widehat{\rho^N}\|_{L^\infty}$ is uniformly bounded the convergence of $\int_0^t [(u_s
  \ast \rho^N)^2 - \| \rho^N \|_{L^2(\R)}^2] (\varphi) \dd s$ follows from the above arguments together with
  dominated convergence theorem. If $(\tilde{\rho}^N)$ is another approximate
  identity, then $| \widehat{\tilde{\rho}^N} (x) - \widehat{\rho^N} (x) |$
  converges pointwise to 0 and is uniformly bounded, so the independence of
  the limit from the approximate identity follows once more from the dominated
  convergence theorem.
  
  Finally, for $F^N (u) \assign u_s^{\diamond 2} (\varphi) - [(u \ast
  \rho^N)^2 - \| \rho^N \|_{L^2(\R)}^2] (\varphi)$ we can solve the Poisson
  equation directly (strictly speaking we would have to first approximate $ u_s^{\diamond 2}$ by $(u_s \ast \rho^M)^2 -  \| \rho^M \|_{L^2(\R)}^2$, but for simplicity we argue already in the limit $M \to \infty$). We get $F^N = L_0 H^N$ for $H^N = W_2 (h^N)$ with
  \[ \widehat{h^N} (\xi_1, \xi_2) = \frac{(1 - \widehat{\rho^N} (- \xi_1))
     \hat{\varphi} (\xi_1 + \xi_2)}{| 2 \pi \xi |^2} + \frac{\widehat{\rho^N} (- \xi_1)
     (1 - \widehat{\rho^N} (- \xi_2)) \hat{\varphi} (\xi_1 + \xi_2)}{|2 \pi \xi |^2} .
  \]
  Let us concentrate on the first term, the second one being essentially of
  the same form (start by bounding $| \widehat{\rho^N} (- \xi_1) | \leqslant \|
  \rho^N \|_{L^1} \leqslant (1 + \| \rho^N \|_{L^1})$ in that case):
  \begin{align*}
    &\left\| \CF^{-1}\left(\frac{(1 - \widehat{\rho^N} (- \xi_1)) \hat{\varphi} (\xi_1 +  \xi_2)}{|2 \pi \xi |^2}\right) \right\|_{\dot{H}^1_s (\mathbb{R}^2)}^2  \leqslant  \left\| \CF^{-1}\left( \frac{(1 - \widehat{\rho^N} (- \xi_1)) \hat{\varphi} (\xi_1 + \xi_2)}{|2 \pi \xi |^2} \right) \right\|_{\dot{H}^1 (\mathbb{R}^2)}^2\\
    &\hspace{20pt} = \int_{\mathbb{R}^2} \left| \frac{(1 - \widehat{\rho^N} (- \xi_1))
    \hat{\varphi} (\xi_1 + \xi_2)}{|2 \pi \xi |^2} \right|^2 | 2 \pi \xi |^2 \dd \xi \\
    &\hspace{20pt}\leqslant \int_{\mathbb{R}^2} \min \{ \| \partial_x \widehat{\rho^N}
    \|_{L^{\infty}} | \xi_1 |, (1 + \| \rho^N \|_{L^1}) \}^2 \frac{|
    \hat{\varphi} (\xi_1 + \xi_2) |^2}{| \xi_1 |^2} \dd \xi\\
    &\hspace{20pt}\leqslant 2 \int^{(1 + \| \rho^N \|_{L^1}) \| \partial_x  \widehat{\rho^N} \|^{- 1}_{L^{\infty}}}_0 \| \partial_x \widehat{\rho^N}
    \|_{L^{\infty}}^2 \dd \xi_1  \| \varphi
    \|_{L^2}^2 \\
    &\hspace{20pt}\quad +2 \int_{(1 + \| \rho^N \|_{L^1}) \|
    \partial_x \hat{\rho}_n \|^{- 1}_{L^{\infty}}}^{\infty} \frac{(1 + \|
    \rho^N \|_{L^1})^2}{| \xi_1 |^2} \dd \xi_1  \| \varphi
    \|_{L^2}^2\\
    &\hspace{20pt} \simeq (1 + \| \rho^N \|_{L^1}) \| \partial_x \widehat{\rho^N} \|_{L^{\infty}}
    \| \varphi \|_{L^2(\R)}^2 \lesssim (1 + \| \rho^N \|_{L^1}) \| x \rho^N \|_{L^1} \| \varphi \|_{L^2}^2 .
  \end{align*}
  Now $\|\rho^N\|_{L^1} = \|\rho\|_{L^1}$ and $\| x \rho^N\|_{L^1} = N^{-1} \|x \rho\|_{L^1}$, and therefore our claim follows from Proposition~\ref{prop:ito-trick} together with Remark~\ref{rmk:ito with finite chaos} and Lemma~\ref{lem:H1 chaos}.
\end{proof}

Proposition~\ref{prop:burgers drift} about the Burgers drift follows by setting
\[
   \int_0^t \partial_x u_s^2 \dd s (\varphi) := \int_0^t u_s^{\diamond 2} \dd s (-\partial_x \varphi).
\]

\begin{remark}
  For the Ornstein--Uhlenbeck process $(X,0) \in \CQ$ one can check that the process
  $\int_0^{\cdummy} X_s^{\diamond 2} \dd s$ has regularity $C([0,T], C^{1 -}_{\mathrm{loc}}(\R))$. But given the bounds of Proposition~\ref{prop:kpz drift} we cannot even evaluate $\int_0^{\cdummy} X_s^{\diamond 2} \dd s$
  in a point, because the Dirac delta just fails to be in $H^{- 1 / 2}(\R)$. The reason is that the martingale argument on which our proof is based gives us at least regularity $1/2-$ in time, and this prevents us from getting better space regularity. On the other hand we are able to increase the time regularity by an
  interpolation argument as shown in the following corollary which will be useful for controlling certain Young integrals below.
\end{remark}

\begin{corollary}\label{cor:kpz drift hoelder}
   For all $\varphi \in \CS(\R)$ and all $T > 0$,
  the process $\left( \left( \int_0^{\cdummy} u_s^{\diamond 2} \dd s
  \right) (\varphi) \right)_{t \in [0, T]}$ is almost surely in $C^{3 / 4 -}
  ([0, T], \mathbb{R})$ and we have for all $\rho \in \CS(\R)$ with $\hat{\rho}
  (0) = 1$ and for all $\alpha < 3 / 4$
  \[ \lim_{n \rightarrow \infty} \left\| \int_0^{\cdummy} u_s^{\diamond 2}
     \dd s (\varphi) - \int_0^{\cdummy} [(u_s \ast \rho^N)^2 - \| \rho^N
     \|_{L^2(\R)}^2] (\varphi) \dd s \right\|_{C^{\alpha} ([0, T],
     \mathbb{R})} = 0, \]
  where the convergence is in $L^p$ for all $p > 0$ and we write again $\rho^N \assign N \rho(N \cdot)$.
\end{corollary}

\begin{proof}
  Proposition~\ref{prop:kpz drift} yields for all $N$, $0 \leqslant s < t$ and $p \in
  [1, \infty)$
  \begin{equation}
    \label{eq:kpz drift hoelder pr1} \mathbb{E} \left[ \left| \int_s^t
    u_s^{\diamond 2} \dd s (\varphi) - \int_s^t [(u_s \ast \rho^N)^2 - \|
    \rho^N \|_{L^2(\R)}^2] (\varphi) \dd s \right|^p \right] \lesssim_{\rho,
    \varphi} | t - s |^{p / 2} N^{- p / 2},
  \end{equation}
  and a direct estimate gives
  \begin{equation}
    \label{eq:kpz drift hoelder pr2} \mathbb{E} \left[ \left| \int_s^t
    [(u_s \ast \rho^N)^2 - \| \rho^N \|_{L^2(\R)}^2] (\varphi) \dd s \right|^p
    \right] \lesssim | t - s |^p \mathbb{E} [| [(u_0 \ast \rho^N)^2 - \|
    \rho^N \|_{L^2}^2] (\varphi) |^2]^{p / 2} .
  \end{equation}
  The expectation on the right hand side is
  \begin{align*}
    &\mathbb{E} [| [(u_0 \ast \rho^N)^2 - \| \rho^N \|_{L^2}^2] (\varphi) |^2]
     =\mathbb{E} \left[ \left| \int_{\mathbb{R}} W_2 (\rho^N (x -
    \cdummy)^{\otimes 2})(u_0) \varphi (x) \dd x \right|^2 \right]\\
    &\hspace{70pt} = \int_{\mathbb{R}} \dd x \int_{\mathbb{R}} \dd x' \mathbb{E}
    [W_2 (\rho^N (x - \cdummy)^{\otimes 2})(u_0) W_2 (\rho^N (x' -
    \cdummy)^{\otimes 2})(u_0) ] \varphi (x) \varphi (x')\\
    &\hspace{70pt} \simeq \int_{\mathbb{R}} \dd x \int_{\mathbb{R}} \dd x'
    \int_{\mathbb{R}^2} \dd y_1 \dd y_2 \rho^N (x - y_1) \rho^N (x -
    y_2) \rho^N (x' - y_1) \rho^N (x' - y_2) \varphi (x) \varphi (x')\\
    &\hspace{70pt} = \int_{\mathbb{R}} \dd x \int_{\mathbb{R}} \dd x' |(\rho^N \ast \rho^N) (x -x')|^2 \varphi (x) \varphi (x') \leqslant \| (\rho^N \ast \rho^N)^2 \ast \varphi \|_{L^2} \| \varphi \|_{L^2} \\
    &\hspace{70pt} \leqslant  \| (\rho^N \ast \rho^N)^2 \|_{L^1} \| \varphi \|_{L^2}^2 =  \| (\rho \ast \rho)^N\|_{L^2}^2 \| \varphi \|_{L^2}^2 = N \| \rho \ast \rho \|_{L^2}^2 \| \varphi \|_{L^2}^2.
  \end{align*}
  Plugging this into~(\ref{eq:kpz drift hoelder pr2}) we get for all $N$
  and $0 \leqslant s < t$
  \begin{equation}
    \label{eq:kpz drift hoelder pr3} \mathbb{E} \left[ \left| \int_s^t
    [(u_s \ast \rho^N)^2 - \| \rho^N \|_{L^2}^2] (\varphi) \dd s \right|^p
    \right] \lesssim_{\rho, \varphi} | t - s |^p N^{p / 2} .
  \end{equation}
  Now if $0 \leqslant s < t \leqslant s + 1$ we apply~(\ref{eq:kpz drift
  hoelder pr1}) and~(\ref{eq:kpz drift hoelder pr3}) with $N \simeq | t -
  s |^{- 1 / 2}$ and get $\mathbb{E} \left[ \left| \int_s^t u_s^{\diamond 2}
  \dd s (\varphi) \right|^p \right] \lesssim | t - s |^{3 p / 4}$, from
  where Kolmogorov's continuity criterion yields the local
  H{\"o}lder-continuity of order $3 / 4 -$. Moreover, for $N > |
  t - s |^{- 1 / 2}$ equation~(\ref{eq:kpz drift hoelder pr1}) gives for
  all $\lambda \in [0, 1]$
  \[ \mathbb{E} \left[ \left| \int_s^t u_s^{\diamond 2} \dd s (\varphi) -
     \int_s^t [(u_s \ast \rho^N)^2 - \| \rho^N \|_{L^2}^2] (\varphi) \dd s
     \right|^p \right] \lesssim_{\rho, \varphi} | t - s |^{p (1 / 2 + (1 -
     \lambda) / 4)} N^{- \lambda p / 2}, \]
  while for $N \leqslant | t - s |^{- 1 / 2}$ we get
  \begin{align*}
    &\mathbb{E} \left[ \left| \int_s^t u_s^{\diamond 2} \dd s (\varphi) -
    \int_s^t [(u_s \ast \rho^N)^2 - \| \rho^N \|_{L^2}^2] (\varphi) \dd s
    \right|^p \right]\\
    &\hspace{10pt} \lesssim_{\rho, \varphi} \{ | t - s |^{p / 2} N^{- p / 2} \}^{\lambda}
    \times \left\{ \mathbb{E} \left[ \left| \int_s^t u_s^{\diamond 2} \dd
    s (\varphi) \right|^p \right] +\mathbb{E} \left[ \left| \int_s^t [(u_s
    \ast \rho^N)^2 - \| \rho^N \|_{L^2}^2] (\varphi) \dd s \right|^p
    \right] \right\}^{1 - \lambda}\\
    &\hspace{10pt}\lesssim_{\rho, \varphi} \{ | t - s |^{p / 2} N^{- p / 2} \}^{\lambda}
    \times \{ | t - s |^{3 p / 4} + | t - s |^p N^{p / 2} \}^{1 - \lambda}\\
    &\hspace{10pt}\leqslant \{ | t - s |^{p / 2} N^{- p / 2} \}^{\lambda} \times \{ | t - s
    |^{3 p / 4} \}^{1 - \lambda},
  \end{align*}
  so that choosing $\lambda > 0$  small and applying once more
  Kolmogorov's continuity criterion we get the convergence of $\int_s^t [(u_s \ast \rho^N)^2 - \| \rho^N \|_{L^2}^2] (\varphi) \dd s$ to $\int_s^t u_s^{\diamond 2} \dd s (\varphi)$ in $L^p(\Omega; C^{\alpha} ([0, T], \mathbb{R}))$ for all $\alpha < 3 / 4$ and all $p > 0$.
\end{proof}

\section{Proof of the main results}\label{sec:proof}

\subsection{Mapping to the Stochastic Heat Equation}\label{sec:map-SHE}

Let $W$ be a space-time white noise on the filtered probability space
  $(\Omega, \mathcal{F}, (\mathcal{F}_t)_{t \geqslant 0}, \mathbb{P})$ and
  let $\eta$ be a $\mathcal{F}_0$-measurable space white noise. Let $(u,\mathcal A) \in \CQ(W,\eta)$ be a strong stationary solution to the stochastic Burgers equation
\begin{equation}\label{eq:SHE-Burgers}
   \dd u_t = \Delta u_t \dd t + \lambda \partial_x u^2_t \dd t +
  \sqrt{2} \partial_x \dd W_t, \qquad u_0 = \eta.
\end{equation}
Our aim is to show that $u$ is unique up to indistinguishability, that is to prove the first part of Theorem~\ref{thm:burgers-uniqueness}. The case $\lambda = 0$ is well understood so from now on let $\lambda \neq 0$. The
basic strategy is to integrate $u$ in the space variable and then to
exponentiate the integral and to show that the resulting process solves (a variant of)
the linear stochastic heat equation. However, it is not immediately obvious
how to perform the integration in such a way that we obtain a useful integral
process. Note that any integral $h$ of $u$ is determined uniquely by its
derivative $u$ and the value $h (\sigma)$ for a test function $\sigma$ with
$\int_{\mathbb{R}} \sigma (x) \dd x = 1$. So the idea, inspired by~\cite{bib:funakiQuastel}, is to fix one such test function and to consider the integral $h$ with $h_t (\sigma) \equiv 0$ for all $t \geqslant 0$.

More concretely, we take $\sigma \in C^{\infty}_c (\mathbb{R})$ with $\sigma
\geqslant 0$ and $\int_{\mathbb{R}} \sigma (x) \dd x = 1$ and consider the
function
\begin{align}\label{eq:Theta-def} \nonumber
  \Theta_x (z) & \assign \int_{\mathbb{R}} \Theta_{x, y} (z) \sigma (y) \dd
  y \assign \int_{\mathbb{R}} (\1_{y \leqslant z \leqslant x}
  -\1_{x < z < y}) \sigma (y) \dd y \\
  & =\1_{z \leqslant x} - \int_{\mathbb{R}} (\1_{y > z}
  \1_{z \leqslant x} +\1_{x < z} \1_{z < y}) \sigma
  (y) \dd y =\1_{(- \infty, x]} (z) - \int_z^{\infty} \sigma (y)
  \dd y,
\end{align}
which satisfies for any $\varphi \in \CS(\R)$
\[ \partial_x \langle \Theta_x, \varphi \rangle_{L^2(\R)} = \partial_x \left(
   \int_{\mathbb{R}} \left( \int_{[y, x]} \varphi (z) \dd z - \int_{ (x,
   y)} \varphi (z) \dd z \right) \sigma (y) \dd y \right) =
   \int_{\mathbb{R}} \varphi (x) \sigma (y) \dd y = \varphi (x) \]
and
\[ \int_{\mathbb{R}} \langle \Theta_x, \varphi \rangle_{L^2(\R)} \sigma (x)
   \dd x = \int_{\mathbb{R}} \int_{\mathbb{R}} \int_{\mathbb{R}}
   (\1_{y \leqslant z \leqslant x} -\1_{x < z < y}) \sigma
   (y) \sigma (x) \dd y \dd x \varphi (z) \dd z = 0, \]
so $x \mapsto \langle \Theta_x, \varphi \rangle_{L^2}(\R)$ is the unique integral
of $\varphi$ which vanishes when tested against $\sigma$. Moreover,
\begin{equation}\label{eq:Theta-Deriv} \langle \Theta_x, \partial_z \varphi \rangle_{L^2(\R)} =
  \int_{\mathbb{R}} (\1_{y \leqslant x} (\varphi (x) - \varphi (y))
  -\1_{x < y} (\varphi (y) - \varphi (x))) \sigma (y) \dd y =
  \varphi (x) - \langle \varphi, \sigma \rangle_{L^2(\R)},
\end{equation}
and in particular
\begin{equation}\label{eq:Laplacian-relation}
  \Delta_x \langle \Theta_x, \varphi
  \rangle_{L^2(\R)} = \partial_x \varphi (x) = \langle \Theta_x, \Delta \varphi
  \rangle_{L^2(\R)} + \langle \partial_z \varphi, \sigma \rangle_{L^2_z(\R)},
\end{equation}
where the notation $L^2_{z}(\R)$ means that the $L^2(\R)$-norm is taken in the variable $z$. 
Let now $\rho \in \CS (\mathbb{R})$ be an even function with $\hat{\rho} \in
C^\infty_c (\mathbb{R})$ and such that $\hat{\rho} \equiv 1$ on a
neighborhood of $0$. We write $\rho^L (x) \assign L \rho (Lx)$, and by our
assumptions on $\rho$ there exists for every $L \in \mathbb{N}$ an
$M \in \mathbb{N}$ such that $\rho^N \ast \rho^L = \rho^L$ for all $N
\geqslant M$. This will turn out to be convenient later. We also write
$\rho^L_x (z) \assign \rho^L (x - z)$ and $\Theta^L_x \assign \rho^L \ast
\Theta_x \in \CS(\R)$ and define
\[ u^L_t (x) \assign (u_t \ast \rho^L) (x) = u_t (\rho^L_x), \qquad h^L_t (x)
   \assign \int_{\mathbb{R}} \Theta_x (z) u^L_t (z) \dd z, \qquad x \in
   \mathbb{R}. \]
Using that $\rho^L$ is even we get $h^L_t (x) = u_t
(\Theta^L_x)$. So since $u$ is a strong stationary solution
to~(\ref{eq:SHE-Burgers}) we have
\[ \dd h^L_t (x) = u_t (\Delta_z \Theta^L_x) \dd t + \dd
   \mathcal{A}_t (\Theta^L_x) + \sqrt{2} \dd W_t (-\partial_z \Theta^L_x), \qquad
   \dd [ h^L (x) ]_t = 2 \| \partial_z \Theta^L_x \|_{L^2(\R)}^2
   \dd t, \]
and~(\ref{eq:Theta-Deriv}) yields
\begin{equation}
  \label{eq:ThetaL-Deriv} \partial_z \Theta^L_x (z) = \int_{\mathbb{R}}
  \Theta_x (z') \partial_z \rho^L (z - z') \dd z' = - \rho^L_z (x) +
  \langle \rho^L_z, \sigma \rangle_{L^2(\R)} = - \rho^L_x (z) + (\rho^L \ast
  \sigma) (z) .
\end{equation}
Set now $\phi^L_t (x) \assign e^{\lambda h^L_t (x)}$. Then the It{\^o} formula for
Dirichlet processes of~\cite{bib:russoVallois} gives
\begin{align*}
  \dd \phi^L_t (x) & = \phi^L_t (x) \left( \lambda \dd h^L_t (x) +
  \frac{1}{2} \lambda^2 \dd [ h^L (x) ]_t \right)\\
  & = \lambda \phi^L_t (x) \left( u^L_t (\Delta_z \Theta^L_x) \dd t + \dd
  \mathcal{A}_t (\Theta^L_x) + \sqrt{2} \dd W_t (-\partial_z \Theta^L_x) + \lambda \|
  \rho^L_x - \rho^L \ast \sigma \|_{L^2(\R)}^2 \dd t \right),
\end{align*}
and from~(\ref{eq:Laplacian-relation}) we get
\begin{align*}
  \Delta_z \Theta^L_x (z) & = \langle \Theta_x, \Delta_z \rho^L_z \rangle_{L^2(\R)}
  = \langle \Theta_x, \Delta_{z'} \rho^L_z \rangle_{L^2_{z'}(\R)} = \Delta_x \langle \Theta_x, \rho^L_z \rangle_{L^2(\R)} - \langle \partial_{z'} \rho^L_z,
  \sigma \rangle_{L^2_{z'}(\R)} \\
  &= \Delta_x \Theta^L_x (z) + \partial_z (\rho^L \ast \sigma) (z).
\end{align*}
Therefore, $u^L_t (\Delta_z \Theta^L_x) = \Delta_x h^L_t (x) + u_t (\rho^L
\ast \partial_z \sigma)$. Since moreover
\[ \lambda \phi^L_t (x) \Delta_x h^L_t (x) = \Delta_x \phi^L_t (x) - \lambda^2
   \phi^L_t (x) (\partial_x h^L_t (x))^2, \]
we obtain
\begin{align*}
  \dd \phi^L_t (x) & = \Delta_x \phi^L_t (x) \dd t + \sqrt{2} \lambda \phi^L_t (x) \dd W_t (-\partial_z \Theta^L_x) \\
  &\quad +\lambda^2 \phi^L_t
  (x) (\lambda^{- 1} \dd \mathcal{A}_t (\Theta^L_x) + (\lambda^{- 1} u_t
  (\rho^L \ast \partial_z \sigma) - (u_t^L (x))^2 + \| \rho^L_x - \rho^L \ast
  \sigma \|_{L^2(\R)}^2) \dd t)
\end{align*}
Expanding the $L^2(\R)$-inner product and noting that $W (-\partial_z \Theta^L_x) = W (\rho^L_x) - W
(\rho^L \ast \sigma)$ by~(\ref{eq:ThetaL-Deriv}), we deduce that
\begin{align}\label{eq:PhiL-x} \nonumber
   \dd \phi^L_t (x) & = \Delta_x \phi^L_t (x) \dd t +
  \sqrt{2} \lambda \phi^L_t (x) \dd W_t (\rho^L_x) + \lambda^2 \dd
  R^L_t (x) + \lambda^2 K^L_x \phi^L(x) \dd t + \lambda^2 \phi^L_t (x) \dd Q^L_t\\
  &\quad - 2 \lambda^2 \phi^L_t (x) \langle \rho^L_x, \rho^L \ast \sigma
  \rangle_{L^2(\R)} \dd t -  \sqrt{2} \lambda \phi^L_t (x) \dd W_t (\rho^L \ast
  \sigma),
\end{align}
where we introduced the processes
\[ R^L_t (x) \assign \int_0^t \phi^L_s (x) \{ \lambda^{- 1} \dd
   \mathcal{A}_s (\Theta^L_x) - ((u_s^L (x))^2 - \langle (u^L_s)^2, \sigma
   \rangle_{L^2(\R)}) \dd s - K^L_x \dd s \} \]
for a deterministic function $K^L$ to be determined, and
\[
  Q^L_t \assign \int_0^t \{ - \langle (u^L_s)^2 - \| \rho^L \|_{L^2(\R)}^2, \sigma
  \rangle_{L^2(\R)} + \| \rho^L \ast \sigma \|_{L^2(\R)}^2 + \lambda^{- 1} u_s (\rho^L
  \ast \partial_z \sigma) \} \dd s.
\]
Integrating~\eqref{eq:PhiL-x} against $\varphi \in C^{\infty}_c (\mathbb{R})$ we get
\begin{align}\label{eq:PhiL}\nonumber
  \dd \phi^L_t (\varphi) & = \phi^L_t (\Delta \varphi) \dd t + \sqrt{2}
  \lambda \int_{\mathbb{R}} \phi^L_t (x) \varphi (x) \dd W_t (\rho^L_x)
  \dd x + \lambda^2 \dd R^L_t (\varphi) + \lambda^2 \phi^L_t( K^L \varphi) + \lambda^2 \phi^L_t (\varphi)
  \dd Q^L_t\\
  & \quad - 2 \lambda^2 \int_{\mathbb{R}} \phi^L_t (x) \varphi (x) \langle \rho^L_x,
  \rho^L \ast \sigma \rangle_{L^2(\R)} \dd x \dd t - \sqrt{2}\lambda \phi^L_t
  (\varphi) \dd W_t (\rho^L \ast \sigma) .
\end{align}

In Section~\ref{sec:convergence} below we will prove the following three lemmas.

\begin{lemma}\label{lem:remainder}
   We have for all $T > 0$, $p > 2$ and all $\varphi \in
  C^{\infty}_c (\mathbb{R})$
  \[ \lim_{L \rightarrow \infty} (\mathbb{E} [\| R^L (\varphi) \|_{p -
     \tmop{var} ; [0, T]}^2] +\mathbb{E} [\sup_{t \leqslant T} | R^L_t
     (\varphi) |^2]) = 0. \]
\end{lemma}

\begin{lemma}\label{lem:KL}
   The deterministic function $K^L$ converges to $\lambda^2/12$ as $L \to \infty$ and is uniformly bounded in the sense that $\sup_{L\in \N, x\in \R} |K^L_x| < \infty$.
\end{lemma}

\begin{lemma}\label{lem:constant process}
  For all $T > 0$ the process $(Q^L_t)_{t \in [0, T]}$ converges in probability in $C^{3 / 4 -}([0, T], \mathbb{R})$ to the zero quadratic variation process
 \[
   Q_t \assign - \int_0^t u^{\diamond 2}_s \dd s (\sigma) + \| \sigma \|_{L^2(\R)}^2 t + \lambda^{-1} \int_0^t u_s(\partial_z \sigma) \dd s, \qquad t \in [0,T].
\]
\end{lemma}

With the help of these results it is easy to prove our main theorem.

\begin{proof}[Proof of Theorem~\ref{thm:burgers-uniqueness}]
  Consider the expansion~(\ref{eq:PhiL}) of $\phi^L (\varphi)$. Denoting
  $\phi_t (\varphi) \assign \lim_{L \rightarrow \infty} \phi^L_t (\varphi)$,
  the stochastic integrals converge to
  \[ \sqrt{2} \lambda \int_{[0, t] \times \mathbb{R}} \phi_s (x) \varphi
     (x) \dd W_s (x) \dd x - \sqrt{2} \lambda \int_0^t \phi_s (\varphi)
     \dd W_s (\sigma) \]
  by the stochastic dominated convergence theorem; see~\cite[Proposition IV.2.13]{bib:revuzYor} for a formulation in the
  finite dimensional setting whose proof carries over without problems to our
  situation. Lemma~\ref{lem:remainder} shows that the $p$-variation of $R^L
  (\varphi)$ converges to zero in $L^2 (\mathbb{P})$ whenever $p > 2$.
  Combining this with Lemma~\ref{lem:KL} and Lemma~\ref{lem:constant process}, the $p$-variation of
  $\phi^L (\varphi)$ stays uniformly bounded in $L$ and therefore we can use
  once more that Lemma~\ref{lem:constant process} gives us local convergence in
  $C^{3 / 4 -}$ for $Q^L$ to obtain that $\int_0^{\cdummy} \phi^L_s (\varphi)
  \dd Q^L_s$ converges as a Young integral to $\int_0^{\cdummy} \phi_s
  (\varphi) \dd Q_s$. In conclusion, we get
  \begin{align*}
    \phi_t (\varphi) & = \langle e^{\eta (\Theta_{\cdummy})}, \varphi
    \rangle_{L^2} + \int_0^t \phi_s (\Delta \varphi) \dd s + \sqrt{2}
    \lambda \int_{[0, t] \times \mathbb{R}} \phi_s (x) \varphi (x) \dd W_s
    (x) \dd x + \frac{\lambda^4}{12} \int_0^t \phi_s(\varphi) \dd s \\
    & \quad + \lambda^2 \int_0^t  \phi_s (\varphi) \dd Q_s - 2 \lambda^2 \int_0^t \phi_s (\varphi \sigma) \dd s - \sqrt{2} \lambda
    \int_0^t \phi_s (\varphi) \dd W_s (\sigma) .
  \end{align*}
  
  So let us define 
  \[
     X_t \assign \sqrt{2} \lambda W_t (\sigma) + \left( - \frac{\lambda^4}{12} +  \lambda^2 \| \sigma \|_{L^2}^2 \right)t - \lambda^2 Q_t ,\qquad Z_t (x) \assign e^{X_t} \phi_t (x), \quad t \geqslant 0,
  \]
  for which $Z_0 (\varphi) = \phi_0 (\varphi)$ and
  \begin{align*}
    \dd Z_t (\varphi) & = e^{X_t} \dd \phi_t (\varphi) + Z_t (\varphi) \dd X_t + \frac{1}{2} Z_t(\varphi) \dd [X]_t + \dd [ \phi (\varphi), e^{X} ]_t \\
    & = Z_t (\Delta \varphi) \dd t + \sqrt{2} \lambda \int_{\mathbb{R}}
    Z_t (x) \varphi (x) \dd W_t (x) \dd x + \frac{\lambda^4}{12} Z_t(\varphi) \dd t + \lambda^2 Z_t (\varphi)
    \dd Q_t \\
    &\quad - 2 \lambda^2 Z_t (\varphi \sigma) \dd t - \sqrt{2} \lambda
    Z_t \dd W_t (\sigma)\\
    &\quad + \sqrt{2} \lambda Z_t (\varphi) \dd W_t (\sigma) + Z_t (\varphi) \left( -\frac{\lambda^4}{12} + \lambda^2 \| \sigma \|_{L^2}^2 \right)
    \dd t - \lambda^2 Z_t (\varphi) \dd Q_t \\
    &\quad + \lambda^2 Z_t(\varphi) \| \sigma\|_{L^2}^2 \dd t \\
    &\quad + 2 \lambda^2 \langle \phi_t \varphi, e^{X_t} \sigma \rangle_{L^2}
    \dd t - 2 \lambda^2 Z_t (\varphi) \| \sigma \|_{L^2}^2 \dd t\\
    & = Z_t (\Delta \varphi) \dd t + \sqrt{2} \lambda \int_{\mathbb{R}}
    Z_t (x) \varphi (x) \dd W_t (x) \dd x.
  \end{align*}
  This means that $Z$ is the unique up to indistinguishability solution to the stochastic heat equation
  \[
     Z_t(\varphi) = \phi_0(\varphi) + \int_0^t Z_s (\Delta \varphi) \dd s +  \sqrt{2} \lambda \int_{[0,t] \times \mathbb{R}}
    Z_s (x) \varphi (x) \dd W_s (x) \dd x.
  \]
  But we know that
  \[ u = \lim_{L \rightarrow \infty} u^L = \lim_{L \rightarrow \infty}
     \partial_x h^L = \lim_{L \rightarrow \infty} \partial_x \lambda^{- 1}
     \log \phi^L = \partial_x \lambda^{- 1} \log \phi ,
  \]
  where the derivative is taken in the distributional sense. Since for fixed $t \geqslant 0$ we have $Z_t (x) = \phi_t (x) e^{X_t}$ and $X_t$ does not depend on the space variable, we get
  \[ \partial_x (\log Z_t (x)) = \partial_x (\log \phi_t (x) + X_t) =
     \partial_x \log \phi_t (x), \]
  and therefore the strong stationary solution $u$ is unique up to indistinguishability.
  
  The uniqueness in law of energy solutions follows in the same way from the weak uniqueness of $Z$.
\end{proof}

\begin{proof}[Proof of Theorem~\ref{thm:kpz-uniqueness}]
   Let $(h, \mathcal{B}) \in \CQ^{\mathrm{KPZ}}$ be a strong almost-stationary solution to the KPZ equation
   \[
     \dd h_t = \Delta h_t \dd t + \lambda ( (\partial_x h_t)^2 - \infty) \dd t + \sqrt{2} \dd W_t, \qquad h_0 = \chi.
   \]
   Since by definition of the pair $(u,\mathcal{A}) \in \CQ(W,\eta)$ we have $u = \partial_x h$ and $\partial_x u(\Theta_x) = u(x)$, we get
   \begin{equation}\label{eq:h vs uTheta}
      h_t(x) = u_t(\Theta_x) + h_t(\sigma),\qquad t \geqslant 0, x \in \R.
   \end{equation}
   Again by definition $u$ is a strong stationary solution to the stochastic Burgers equation. So in the proof of Theorem~\ref{thm:burgers-uniqueness} we showed that
   \begin{equation}\label{eq:uTheta vs CH}
      u_t(\Theta_x) = \lambda^{-1} \log \phi_t(x) = \lambda^{-1} (\log \tilde{Z}_t(x) - X_t),
   \end{equation}
   where $\tilde{Z}$ solves the linear multiplicative heat equation with initial condition $e^{\lambda u_0(\Theta_\cdot)}$ and $X$ was defined by $X_0 = 0$ and
   \begin{align*}
      \dd X_t & = \sqrt{2} \lambda \dd W_t(\sigma) + \left(- \frac{\lambda^4}{12} +  \lambda^2 \| \sigma \|_{L^2}^2 \right) \dd t - \lambda^2 (- u^{\diamond 2}_t (\sigma) \dd t + \| \sigma \|_{L^2}^2 \dd t + \lambda^{-1} u_t(\partial_z \sigma) \dd t)\\
      & = \sqrt{2} \lambda \dd W_t(\sigma) - \frac{\lambda^4}{12} \dd t + \lambda^2 u^{\diamond 2}_t (\sigma) \dd t - \lambda u_t(\partial_z \sigma) \dd t \\
      & = \lambda (h_t(\Delta \sigma) \dd t + \lambda u^{\diamond 2}_t (\sigma) \dd t + \sqrt{2} \dd W_t(\sigma)) - \frac{\lambda^4}{12} \dd t = \lambda \dd h_t(\sigma) - \frac{\lambda^4}{12} \dd t.
   \end{align*}
   This shows that $\lambda^{-1} X_t =  h_t(\sigma) -  h_0(\sigma) - \frac{\lambda^3}{12}t$, and therefore we get with~\eqref{eq:h vs uTheta} and~\eqref{eq:uTheta vs CH}
   \[
      h_t(x) = \lambda^{-1} (\log \tilde{Z}_t(x) - X_t) + h_t(\sigma) = \lambda^{-1} \log \tilde{Z}_t(x) + h_0(\sigma) + \frac{\lambda^3}{12}t.
   \]
   Finally, $\tilde{Z}_t = e^{-\lambda \chi(\sigma)} Z_t$, where $Z$ solves~\eqref{eq:thm-KPZ-SHE}, the linear multiplicative heat equation with initial condition $e^{\lambda \chi}$, and this concludes the proof of the strong uniqueness for strong almost-stationary solutions. The weak uniqueness of energy solutions follows from the weak uniqueness of $Z$.
\end{proof}

\subsection{Convergence of the remainder terms}\label{sec:convergence}

We now proceed to prove Lemmas~\ref{lem:remainder}-\ref{lem:constant
process} on the convergence of $R^L$, $K^L$ and $Q^L$, respectively.

\subsubsection{Proof of Lemmas~\ref{lem:remainder} and~\ref{lem:KL}}

To treat $R^L$ we introduce the auxiliary process
\[ R^{L, N}_t (x) \assign \int_0^t \phi^L_s (x) \{ (- (\rho^N \ast u_s)^2
   (\partial_z \Theta^L_x)) \dd s - ((u_s^L (x))^2 - \langle (u^L_s)^2,
   \sigma \rangle_{L^2}) \dd s - K^{L, N}_x \dd s \} \]
for $K^{L, N}_x$ that will be determined below. Since by assumption $\mathcal{A} = \int_0^\cdot \partial_x u_s^2 \dd s$ and $\partial_z \Theta^L_x \in \CS(\R)$ is a nice test function, we get from Corollary~\ref{cor:kpz drift hoelder} that $R^{L,N}_t(\varphi)$ converges to $R^{L}_t(\varphi)$ in $L^2(\P)$.

We also define
\[ r^{L, N} (u_s, x) \assign \phi_s^L (x) ((- (\rho^N \ast u_s)^2 (\partial_z
   \Theta^L_x)) - ((u_s^L (x))^2 - \langle (u^L_s)^2, \sigma \rangle_{L^2}) -
   K^{L, N}_x) \]
so that $R_t^{L, N} (x) = \int_0^t r^{L, N} (u_s, x) \dd s$. Using
Corollary~\ref{cor:KV} we can estimate for $\varphi \in C^{\infty}_c
(\mathbb{R})$
\begin{equation}
  \mathbb{E} [\| R^{L,N} (\varphi) \|_{p - \tmop{var} ; [0, T]}^2] +\mathbb{E}
  [\sup_{0 \leqslant t \leqslant T} | R^{L, N}_t (\varphi) |^2] \lesssim (T \vee T^{2}) \|
  r^{L, N} (\cdot, \varphi) \|_{- 1}^2,
\end{equation}
where we recall that
\[ \| r^{L, N} (\cdot, \varphi) \|_{- 1}^2 = \sup_{F \in \CC} \{2\mathbb{E}
   [r^{L, N} (u_0, \varphi) F (u_0)] - \| F \|_1^2\}, \]
where $\CC$ are the cylinder functions and $\| F \|_1^2 = 2 \mathbb{E} [\| \partial_x \dD F (u_0) \|_{L^2
(\mathbb{R})}^2]$ in terms of the Malliavin derivative $\dD$ associated to
the measure $\mu$. We prove below that we can choose $K^{L, N}_x$ so that
$\mathbb{E} [r^{L, N} (u_t, x)] =\mathbb{E} [r^{L, N} (u_0, x)] = 0$ for all
$x \in \mathbb{R}$. This is necessary in order for $\| r^{L, N} (\cdot,
\varphi) \|_{- 1}$ to be finite for all $\varphi \in C^{\infty}_c
(\mathbb{R})$. At this point everything boils down to controlling
$\mathbb{E} [r^{L, N} (u_0, \varphi) F (u_0)]$ and to showing that it goes
to zero as first $N \rightarrow \infty$ and then $L \rightarrow \infty$. 

Observe that the random variable $(- (\rho^N \ast u_0)^2 (\partial_z
\Theta^L_x)) - ((u_0^L (x))^2 - \langle (u^L_0)^2, \sigma
\rangle_{L^2})$ is an element of the second homogeneous chaos of $u_0$. Let us compute its
kernel. From~(\ref{eq:Theta-Deriv}) we get
\begin{align*}
  - (\rho^N \ast u_0)^2 (\partial_z \Theta^L_x) &= \int_{\mathbb{R}^2} \left[
  - \int_{\mathbb{R}} \dd z \partial_z \Theta^L_x (z) \rho^N (z - y_1)
  \rho^N (z - y_2) \right] W (\dd y_1 \dd y_2)\\
  &= \int_{\mathbb{R}^2} \left[ \int_{\mathbb{R}} \dd z \int_{\mathbb{R}}
  \dd z' \Theta_x (z') \partial_{z'} \rho^L (z - z') \rho^N_z (y_1)
  \rho^N_z (y_2) \right] W (\dd y_1 \dd y_2)\\
  &= \int_{\mathbb{R}^2} \left[ \int_{\mathbb{R}} \dd z (\rho^L (z - x) -
  \langle \rho^L (z - \cdummy), \sigma \rangle_{L^2 (\mathbb{R})}) \rho^N_z
  (y_1) \rho^N_z (y_2) \right] W (\dd y_1 \dd y_2)\\
  &= \int_{\mathbb{R}^2} \left[ \int_{\mathbb{R}} \dd z (\rho^L_x (z) -
  \langle \rho^L_z, \sigma \rangle_{L^2 (\mathbb{R})}) \rho^N_z (y_1)
  \rho^N_z (y_2) \right] W (\dd y_1 \dd y_2),
\end{align*}
and furthermore
\[ ((u_0^L (x))^2 - \langle (u^L_0)^2, \sigma \rangle_{L^2}) =
   \int_{\mathbb{R}^2} [\rho^L_x (y_1) \rho^L_x (y_2) - \langle \rho^L_{y_1}
   \rho^L_{y_2}, \sigma \rangle_{L^2 (\mathbb{R})}] W (\dd y_1 \dd y_2)
   . \]
Let therefore
\begin{equation}
  g_x^{L, N} (y_1, y_2) \assign  \int_{\mathbb{R}} \dd z (\rho^L_x
  (z) - \langle \rho^L_z, \sigma \rangle_{L^2 (\mathbb{R})}) \rho^N_z (y_1)
  \rho^N_z \left( y_2 \right) - (\rho^L_x (y_1) \rho^L_x (y_2) - \langle
  \rho^L_{y_1} \rho^L_{y_2}, \sigma \rangle_{L^2 (\mathbb{R})}),
\end{equation}
so that
\[
   W_2 (g_x^{L, N}) = (- (\rho^N \ast u_0)^2 (\partial_z \Theta^L_x)) - ((u_0^L (x))^2 -
   \langle (u^L_0)^2, \sigma \rangle_{L^2}) . \]
We let also $W_1 (g_x^{L, N} (y_1, \cdot)) \assign \int_{\mathbb{R}} g_x^{L,
N} (y_1, y_2) W (\dd y_2)$. Using the partial integration by parts derived in Lemma~\ref{lem:IBP} we are able to bound
$\| r^{L, N} (\cdummy, \varphi) \|_{- 1}$ by a constant:

\begin{lemma}\label{lem:rLN-decomposition}
  Setting $K^{L, N}_x \assign \lambda^2 \int_{\mathbb{R}^2} g_x^{L, N} (y_1,
  y_2) \Theta^L_x (y_1) \Theta^L_x (y_2) \dd y_1 \dd y_2$ we have for
  all $F \in \CC$
  \begin{equation}
    \label{eq:rLN-bound} | \mathbb{E} [r^{L, N} (u_0, \varphi) F (u_0)] |
    \lesssim \| F \|_1 (A_1^{L, N} + C_1^{L, N}),
  \end{equation}
  and in particular $\| r^{L, N} (\cdummy, \varphi) \|_{- 1} \lesssim A_1^{L,
  N} + C_1^{L, N}$, where
  \begin{equation}
    \label{eq:A1-def} A_1^{L, N} \assign \mathbb{E} \left[ \left\|
    \int_{\mathbb{R}} \varphi (x) W_1 (g_x^{L, N} (y_1, \cdot)) \diamond
    \phi_0^L (x) \dd x \right\|_{\dot{H}^{- 1}_{y_1} (\mathbb{R})}^2
    \right]^{1 / 2}
  \end{equation}
  and
  \begin{equation}
    \label{eq:C1-def} C_1^{L, N} \assign \mathbb{E} \left[ \left\|
    \int_{\mathbb{R}} \varphi (x) \phi^L_0 (x) \int_{\mathbb{R}} g_x^{L, N}
    (y_1, y_2) \Theta^L_x (y_1) \dd y_1 \dd x \right\|_{\dot{H}^{-
    1}_{y_2} (\mathbb{R})}^2 \right]^{1 / 2} .
  \end{equation}
  Here the notation $\dot{H}^{\alpha}_y (\mathbb{R})$ means that the norm is
  taken in the $y$-variable and
  \[ W_1 (g_x^{L, N} (y_1, \cdot)) \diamond \phi_0^L (x) \assign W_1 (g_x^{L,
     N} (y_1, \cdot)) \phi_0^L (x) - \int_{\mathbb{R}} g_x^{L, N} (y_1, y_2)
     \dD_{y_2} \phi_0^L (x) \dd y_2 \]
  is a partial Wick contraction in the sense that $\mathbb{E} [W_1 (g_x^{L,
  N} (y_1, \cdot)) \diamond \phi_0^L (x)] = 0$.
\end{lemma}

\begin{proof}
  Consider
  \begin{equation}
    \mathbb{E} [r^{L, N} (u_0, \varphi) F (u_0)] = \int_{\mathbb{R}}
    \varphi (x) \mathbb{E} [(W_2 (g_x^{L, N}) - K^{L, N}_x) \phi_0^L (x) F
    (u_0)] \dd x.
  \end{equation}
  Partially integrating by parts $W_2 (g_x^{L, N})$, we have
  \begin{align*}
    \mathbb{E} [W_2 (g_x^{L, N}) \phi_0^L (x) F (u_0)] & =
    \int_{\mathbb{R}} \mathbb{E} [W_1 (g_x^{L, N} (y_1, \cdot)) \dD_{y_1}
    [\phi_0^L (x) F (u_0)]] \dd y_1\\
    & = \int_{\mathbb{R}} \mathbb{E} [W_1 (g_x^{L, N} (y_1, \cdot)) \phi_0^L
    (x) \dD_{y_1} F (u_0)] \dd y_1 \\
    &\quad + \int_{\mathbb{R}} \mathbb{E}
    [W_1 (g_x^{L, N} (y_1, \cdot)) \dD_{y_1} (\phi_0^L (x)) F (u_0)]
    \dd y_1 .
  \end{align*}
  The second term on the right hand side can be integrated by parts again to obtain
  \begin{align*}
    \int_{\mathbb{R}} \mathbb{E} [W_1 (g_x^{L, N} (y_1, \cdot)) \dD_{y_1}
    (\phi_0^L (x)) F (u_0)] \dd y_1 & = \int_{\mathbb{R}^2} g_x^{L, N}
    (y_1, y_2) \mathbb{E} [\dD^2_{y_1, y_2} (\phi_0^L (x)) F (u_0)]
    \dd y_1 \dd y_2\\
    &\quad + \int_{\mathbb{R}^2} g_x^{L, N} (y_1, y_2) \mathbb{E} [\dD_{y_1}
    (\phi_0^L (x)) \dD_{y_2} F (u_0)] \dd y_1 \dd y_2,
  \end{align*}
  while the first term can be written as
  \begin{align*}
    \int_{\mathbb{R}} \mathbb{E} [W_1 (g_x^{L, N} (y_1, \cdot)) \phi_0^L (x)
    \dD_{y_1} F (u_0)] \dd y_1 & = \int_{\mathbb{R}} \mathbb{E} [(W_1
    (g_x^{L, N} (y_1, \cdot)) \diamond \phi_0^L (x)) \dD_{y_1} F (u_0)]
    \dd y_1\\
    &\quad + \int_{\mathbb{R}^2} g_x^{L, N} (y_1, y_2) \mathbb{E} [\dD_{y_2}
    (\phi_0^L (x)) \dD_{y_1} F (u_0)] \dd y_1 \dd y_2 .
  \end{align*}
  In conclusion, we have the decomposition
  \begin{equation}
     \mathbb{E} [r^{L, N} (u_0, \varphi) F (u_0)] = A^{L, N} + B^{L, N} + C^{L, N}
  \end{equation}
  with
  \[ A^{L, N} \assign \int_{\mathbb{R}} \varphi (x) \int_{\mathbb{R}}
     \mathbb{E} [(W_1 (g_x^{L, N} (y_1, \cdot)) \diamond \phi_0^L (x))
     \dD_{y_1} F (u_0)] \dd y_1 \dd x, \]
  \[ B^{L, N} \assign \int_{\mathbb{R}} \varphi (x) \left[
     \int_{\mathbb{R}^2} g_x^{L, N} (y_1, y_2) \mathbb{E} \left[
     \dD_{y_1, y_2}^2 (\phi_0^L (x)) F \left( u_0 \right) \right] \dd
     y_1 \dd y_2 - K^{L, N}_x \mathbb{E} [^{} \phi_0^L (x) F (u_0)]
     \right] \dd x, \]
  and
  \[ C^{L, N} \assign 2 \int_{\mathbb{R}} \varphi (x) \int_{\mathbb{R}^2}
     g_x^{L, N} (y_1, y_2) \mathbb{E} [\dD_{y_1} (\phi_0^L (x))
     \dD_{y_2} F (u_0)] \dd y_1 \dd y_2 \dd x. \]
  So it suffices to bound the three terms $A^{L, N}$, $B^{L, N}$, $C^{L, N}$
  independently. In order to proceed, observe that
  \[ \dD_{y_1} \phi_0^L (x) = \dD_{y_1} (e^{\lambda u_0 (\Theta^L_x)}) =
     \lambda \phi_0^L (x) \Theta^L_x (y_1), \qquad \dD_{y_1, y_2}^2
     \phi_0^L (x) = \lambda^2 \phi_0^L (x) \Theta^L_x (y_1) \Theta^L_x (y_2),
  \]
  so that by definition of $K^{L, N}_x$
  \begin{align*}
    B^{L, N} & = \int_{\mathbb{R}} \varphi (x) \left[ \int_{\mathbb{R}^2}
    g_x^{L, N} (y_1, y_2) \lambda^2 \Theta^L_x (y_1) \Theta^L_x (y_2) \dd
    y_1 \dd y_2 - K^{L, N}_x \right] \mathbb{E} [\phi_0^L (x) F (u_0)]
    \dd x\\
    & = \int_{\mathbb{R}} \varphi (x) [K^{L, N}_x - K^{L, N}_x] \mathbb{E}
    [\phi_0^L (x) F (u_0)] \dd x = 0
  \end{align*}
  and
  \[ C^{L, N} = 2 \int_{\mathbb{R}} \varphi (x) \int_{\mathbb{R}^2} g_x^{L,
     N} (y_1, y_2) \lambda \Theta^L_x (y_1) \mathbb{E} [\phi_0^L (x)
     \dD_{y_2} F (u_0)] \dd y_1 \dd y_2 \dd x. \]
  Using the duality of $\dot{H}^1 (\mathbb{R})$ and
  $\dot{H}^{- 1} (\mathbb{R})$ and the Cauchy-Schwarz inequality, we bound
  further
  \begin{align*}
     (A^{L, N})^2 & \leqslant \mathbb{E} \left[ \left\| \int_{\mathbb{R}} \varphi (x) (W_1 (g_x^{L, N} (y_1, \cdot)) \diamond \phi_0^L (x)) \dd  x \right\|_{\dot{H}^{- 1}_{y_1} (\mathbb{R})}^2 \right] \mathbb{E} [\|  \dD_{y_1} F (u_0) \|_{\dot{H}^1_{y_1} (\mathbb{R})}^2] \\
     & = (A_1^{L,  N})^2 \| F \|_1^2,
  \end{align*}
  where $A^{L, N}_1$ is the constant defined in~(\ref{eq:A1-def}). Similarly, we obtain
  \begin{align*}
    (C^{L, N})^2 & \lesssim \mathbb{E} \left[ \left\| \int_{\mathbb{R}}
    \varphi (x) \phi^L_0 (x) \int_{\mathbb{R}} g_x^{L, N} (y_1, y_2)
    \Theta^L_x (y_1) \dd y_1 \dd x \right\|_{\dot{H}^{- 1}_{y_2}
    (\mathbb{R})} \right]^2 \mathbb{E} [\| \dD_{y_2} F (u_0)
    \|_{\dot{H}^1_{y_2} (\mathbb{R})}^2]\\
    & = (C_1^{L, N})^2 \| F \|_1^2,
  \end{align*}
  where $C^{L, N}_1$ is the constant in~(\ref{eq:C1-def}). This
  proves~(\ref{eq:rLN-bound}).
\end{proof}

So to control $R^{L, N}$ it remains to show that the two constants $A_1^{L,
N}$ and $C_1^{L, N}$ vanish in the limit $N, L \rightarrow \infty$. Before doing so let us prove Lemma~\ref{lem:KL}. More precisely we show the following refined version:

\begin{lemma}
  We have $\sup_{L, N, x} | K^{L, N}_x | < \infty$ and for all $x \in \mathbb{R}$
  \[
     K^L_x \assign \lim_{N \to \infty} K^{L,N}_x = - ((\rho^L \ast \rho^L \ast \Theta_x)^2 (x) - \langle
    (\rho^L \ast \rho^L \ast \Theta_x)^2, \sigma \rangle_{L^2 (\mathbb{R})}) 
  \]
  and $\lim_{L \rightarrow \infty} K^{L}_x =  \lambda^2 / 12$.
\end{lemma}

\begin{proof}
  Recall that $\Theta^L_x = \Theta_x \ast \rho^L$, so
  \begin{align}\label{eq:KL pr1} \nonumber 
    \lambda^{- 2} K^{L, N}_x & = \int_{\mathbb{R}^2} g_x^{L,
    N} (y_1, y_2) \Theta^L_x (y_1) \Theta^L_x (y_2) \dd y_1 \dd y_2\\ \nonumber
    & = \int_{\mathbb{R}} (\rho^L_x (z) - \langle \rho^L_z, \sigma
    \rangle_{L^2 (\mathbb{R})}) (\rho^N \ast \rho^L \ast \Theta_x)^2 \left( z
    \right) \dd z \\
    &\quad  - ((\rho^L \ast \rho^L \ast \Theta_x)^2 (x) - \langle
    (\rho^L \ast \rho^L \ast \Theta_x)^2, \sigma \rangle_{L^2 (\mathbb{R})}).
  \end{align}
  By~(\ref{eq:ThetaL-Deriv}) we know that for $N \rightarrow \infty$ the first
  term on the right hand side converges to
  \begin{align*}
    \int_{\mathbb{R}} (\rho^L_x (z) - \langle \rho^L_z, \sigma \rangle_{L^2
    (\mathbb{R})}) (\rho^L \ast \Theta_x) (z)^2 \dd z & = -
    \int_{\mathbb{R}} \partial_z (\rho^L \ast \Theta_x) (z) (\rho^L \ast
    \Theta_x)^2 (z) \dd z\\
    & = - \frac{1}{3} \int_{\mathbb{R}} \partial_z (\rho^L \ast \Theta_x)^3 (z)
    \dd z = 0,
  \end{align*}
  where in the last step we used that $\rho^L \ast \Theta_x \in \CS
  (\mathbb{R})$. Moreover,
  \begin{align*}
    &\left| \int_{\mathbb{R}} (\rho^L_x (z) - \langle \rho^L_z, \sigma
    \rangle_{L^2 (\mathbb{R})}) (\rho^N \ast \rho^L \ast \Theta_x)^2 (z)
    \dd z \right|\\
    &\hspace{50pt}= | (\rho^L \ast (\rho^N \ast \rho^L \ast \Theta_x)^2) (x) - \langle
    \rho^L \ast (\rho^N \ast \rho^L \ast \Theta_x)^2, \sigma \rangle_{L^2
    (\mathbb{R})} |\\
    &\hspace{50pt} \leqslant \| \rho^L \|_{L^1 (\mathbb{R})} \| (\rho^N \ast \rho^L \ast
    \Theta_x)^2 \|_{L^{\infty} (\mathbb{R})} (1 + \| \sigma \|_{L^1
    (\mathbb{R})}) \leqslant \| \Theta_x \|_{L^{\infty} (\mathbb{R})}^2 (1 +
    \| \sigma \|_{L^1 (\mathbb{R})}) \lesssim 1,
  \end{align*}
  and by similar arguments also the second term on the right hand side
  of~(\ref{eq:KL pr1}) stays bounded in $L, N, x$. Recalling that $\Theta_x
  (z) =\1_{(- \infty, x]} (z) - \int_z^{\infty} \sigma (y) \dd y$,
  we get by symmetry of $\rho \ast \rho$
  \[ \lim_{L \rightarrow \infty} (\rho^L \ast \rho^L \ast \Theta_x) (x) =
     \lim_{L \rightarrow \infty} (\rho \ast \rho)^L \ast \Theta_x (x) =
     \frac{1}{2} - \int_x^{\infty} \sigma (y) \dd y \]
  as well as $\lim_{L \rightarrow \infty} (\rho^L \ast \rho^L \ast \Theta_x)
  =\1_{(- \infty, x]} - \int_{\cdummy}^{\infty} \sigma (y) \dd y$
  in $L^p (\mathbb{R})$ for any $p \in [1, \infty)$. In particular,
  \begin{align} \label{eq:KL pr2} \nonumber
    &\lim_{L \rightarrow \infty} ((\rho^L \ast \rho^L \ast \Theta_x)^2 (x) -
    \langle (\rho^L \ast \rho^L \ast \Theta_x)^2, \sigma \rangle_{L^2
    (\mathbb{R})})\\
    &\hspace{30pt} = \frac{1}{4} - \int_x^{\infty} \sigma (y) \dd y +
    \left( \int_x^{\infty} \sigma (y) \dd y \right)^2 - \int_{\mathbb{R}}
    \left( \1_{(- \infty, x]} (z) - \int_z^{\infty} \sigma (y) \dd
    y \right)^2 \sigma (z) \dd z.
  \end{align}
  For the last term on the right hand side we further get
  \begin{align*}
    &\int_{\mathbb{R}} \left( \1_{(- \infty, x]} (z) -
    \int_z^{\infty} \sigma (y) \dd y \right)^2 \sigma (z) \dd z\\
    &\hspace{15pt}= \int_{- \infty}^x \sigma (z) \dd z - 2 \int_{\mathbb{R}^2} \dd y
    \dd z\1_{z \leqslant x} \1_{y \geqslant z} \sigma (y)
    \sigma (z) + \int_{\mathbb{R}^3} \dd z \dd y_1 \dd y_2
    \1_{y_1 \geqslant z} \1_{y_2 \geqslant z} \sigma (y_1)
    \sigma (y_2) \sigma (z)\\
    &\hspace{15pt}= 1 - \int_x^{\infty} \sigma (z) \dd z - 2 \int_{\mathbb{R}^2} \dd
    y \dd z\1_{z \leqslant x} \1_{y \geqslant z} \sigma
    (y) \sigma (z) + \int_{\mathbb{R}^3} \dd z \dd y_1 \dd y_2
    \1_{y_1 \geqslant z} \1_{y_2 \geqslant z} \sigma (y_1)
    \sigma (y_2) \sigma (z),
  \end{align*}
  and the three-dimensional integral takes the value
  \[ \int_{\mathbb{R}^3} \dd z \dd y_1 \dd y_2 \1_{y_1
     \geqslant z} \1_{y_2 \geqslant z} \sigma (y_1) \sigma (y_2)
     \sigma (z) = \int_{\mathbb{R}^3} \dd z \dd y_1 \dd y_2
     \1_{z = \min \{ y_1, y_2, z \}} \sigma (y_1) \sigma (y_2) \sigma
     (z) = \frac{1}{3} \]
  by symmetry of the variables $z, y_1, y_2$. To compute the two-dimensional integral observe first that
  \begin{align*}
     ( \1_{z\leqslant x, z \leqslant y} + \1_{y\leqslant x, y < z}) = \1_{z \leqslant y} - \1_{x < z \leqslant y} + \1_{y < z} - \1_{x < y < z} = 1 - \1_{x <y} \1_{x < z},
  \end{align*}
   and integrating this against $\sigma(y) \sigma(z) \dd y \dd z$ and using the symmetry in $(y,z)$, we get
  \begin{equation*}
     \int_{\mathbb{R}^2} \dd y \dd z\1_{z
    \leqslant x} \1_{y \geqslant z} \sigma (y) \sigma (z) =
    \frac{1}{2} \left( 1 - \left( \int_x^{\infty} \sigma (z) \dd z
    \right)^2 \right) .
  \end{equation*}
  Plugging all this back into~(\ref{eq:KL pr2}) we have
  \begin{align*}
    &\lim_{L \rightarrow \infty} ((\rho^L \ast \rho^L \ast \Theta_x)^2 (x) -
    \langle (\rho^L \ast \rho^L \ast \Theta_x)^2, \sigma \rangle_{L^2
    (\mathbb{R})}) \\
    &\hspace{10pt} = \frac{1}{4} - \int_x^{\infty} \sigma (y) \dd y +  \left( \int_x^{\infty} \sigma (y) \dd y \right)^2
    - \left( 1 - \int_x^{\infty} \sigma (z) \dd z -  \left( 1 - \left( \int_x^{\infty} \sigma (z) \dd z
    \right)^2 \right) + \frac{1}{3} \right)\\
    &\hspace{10pt}= \frac{1}{4} - \frac{1}{3} = -
    \frac{1}{12},
  \end{align*}
  which concludes the proof.
\end{proof}

The following computation will be useful for controlling both $A_1^{L,N}$ and $C^{L,N}_1$, which is why we outsource it in a separate lemma.

\begin{lemma}\label{lem:Gxx}
   Define the kernel
   \[ G^{L, N}_{x, x'} (y_1) \assign \int_{\mathbb{R}} \Theta^L_{x'} (y_2)
     g_x^{L, N} (y_1, y_2) \dd y_2 . \]
    Then there exists $C>0$ such that for all $L\in \N$ there is $M(L) \in \N$ with
  \[
     \sup_{x, x' \in \R, N \geqslant M(L)} \| G^{L, N}_{x, x'} \|_{\dot{H}^{- 1} (\mathbb{R})} \leqslant C L^{- 1 / 2}.
  \]
\end{lemma}

\begin{proof}
  We argue by duality. For $f \in C^{\infty}_c$ we have
  \begin{align} \nonumber
    \langle G^{L, N}_{x, x'}, f \rangle_{L^2 (\mathbb{R})} & =
    \int_{\mathbb{R}} \dd z (\rho^L_x (z) - \langle \rho^L_z, \sigma
    \rangle_{L^2 (\mathbb{R})}) (\rho^N \ast f) (z) (\rho^N \ast
    \Theta^L_{x'}) (z)\\ \nonumber
    &\quad - \{ (\rho^L \ast f) (x) (\rho^L \ast \Theta^L_{x'}) (x) - \langle (\rho^L
    \ast f) (\rho^L \ast \Theta^L_{x'}), \sigma \rangle_{L^2 (\mathbb{R})}
    \}\\ \nonumber
    & = (\rho^L \ast ((\rho^N \ast f) (\rho^N \ast \Theta^L_{x'}))) (x) -
    \int_{\mathbb{R}} \dd z \sigma (z) (\rho^L \ast ((\rho^N \ast f)
    (\rho^N \ast \Theta^L_{x'}))) (z)\\ \nonumber
    &\quad - \left\{ (\rho^L \ast f) (x) (\rho^L \ast \Theta^L_{x'}) (x) -
    \int_{\mathbb{R}} \dd z \sigma (z) (\rho^L \ast f) (z) (\rho^L \ast
    \Theta^L_{x'}) (z) \right\}\\ \nonumber
    & =\mathbb{E} [f (x - R^L_1 - R^N_2) \Theta^L_{x'} (x - R^L_1 - R^N_3) - f
    (Z - R^L_1 - R^N_2) \Theta^L_{x'} (Z - R^L_1 - R^N_3)]\\
    &\quad -\mathbb{E} [f (x - R^L_1) \Theta^L_{x'} (x - R^L_4) - f (Z - R^L_1)
    \Theta^L_{x'} (Z - R^L_4)],
  \end{align}
  where the random variables $(R^L_1, R^N_2, R^N_3, R^L_4, Z)$ are independent
  and $R^L_i \sim \rho^L$, $R^N_i \sim \rho^N$, and $Z \sim \sigma$ (note that
  $\rho^L, \rho^N, \sigma$ are all probability densities). The observation
  that we can simplify the notation in this way is taken from~\cite{bib:funakiQuastel}. Note that by assumption $\rho^L \ast \rho^N = \rho^L$ for sufficiently large $N$, and therefore
  \[ \mathbb{E} [\Theta^L_{x'} (x - R^L_1 - R^N_3) - \Theta^L_{x'} (x -
     R^L_4)] = ((\rho^L \ast \rho^N) \ast \Theta^L_{x'}) (x) - (\rho^L \ast
     \Theta^L_{x'}) (x) = 0. \]
   Similarly $\mathbb{E} [f (Z)
  (\Theta^L_{x'} (Z - R^L_1 - R^N_3) - \Theta^L_{x'} (Z - R^L_4))] = 0$ by  the independence of $Z$ and $R^L_i, R^N_i$, and hence we can regroup
  \begin{align} \nonumber
    \langle G^{L, N}_{x, x'}, f \rangle_{L^2 (\mathbb{R})} & =\mathbb{E} [(f
    (x - R^L_1 - R^N_2) - f (x)) (\Theta^L_{x'} (x - R^L_1 - R^N_3) -
    \Theta^L_{x'} (x - R^L_4))]\\ \nonumber
    &\quad +\mathbb{E} [(f (x - R^L_1 - R^N_2) - f (x - R^L_1)) \Theta^L_{x'} (x -
    R^L_4)]\\ \nonumber
    &\quad -\mathbb{E} [(f (Z - R^L_1 - R^N_2) - f (Z)) (\Theta^L_{x'} (Z - R^L_1 -
    R^N_3) - \Theta^L_{x'} (Z - R^L_4))]\\
    &\quad -\mathbb{E} [(f (Z - R^L_1 - R^N_2) - f (Z - R^L_1)) \Theta^L_{x'} (Z -
    R^L_4)] .
  \end{align}
  Let us estimate for example the most complicated term
  \begin{align}\label{eq:A1-exp-bound} \nonumber
    & | \mathbb{E} [(f (Z - R^L_1 - R^N_2) - f (Z)) (\Theta^L_{x'} (Z - R^L_1 -
    R^N_3) - \Theta^L_{x'} (Z - R^L_4))] |\\ \nonumber
    &\hspace{20pt}= \left| \mathbb{E} \left[ \left( \int_Z^{Z - R^L_1 - R^N_2} f' (y)
    \dd y \right) (\Theta^L_{x'} (Z - R^L_1 - R^N_3) - \Theta^L_{x'} (Z -
    R^L_4)) \right] \right|\\
    &\hspace{20pt} \leqslant 2 \| f \|_{\dot{H}^1 (\mathbb{R})} \|
    \Theta^L_{x'} \|_{L^{\infty} (\mathbb{R})} | \mathbb{E} [| R^L_1 + R^N_2
    |^{1 / 2}] | \lesssim \| f \|_{\dot{H}^1 (\mathbb{R})} (L^{- 1 / 2} +
    N^{- 1 / 2}) .
  \end{align}
  The other terms can be controlled using the same arguments, and therefore we get
  \begin{equation}\label{eq:Gxx-bound} \langle G^{L, N}_{x, x'}, f \rangle_{L^2
    (\mathbb{R})} \lesssim \| f \|_{\dot{H}^1 (\mathbb{R})} (L^{- 1 / 2} +
    N^{- 1 / 2}) \lesssim \| f \|_{\dot{H}^1 (\mathbb{R})} L^{- 1 / 2},
  \end{equation}
  which yields $\| G^{L, N}_{x, x'} \|_{\dot{H}^{- 1} (\mathbb{R})} \lesssim
  L^{- 1 / 2}$ by the density of $C^{\infty}_c (\mathbb{R})$ in $\dot{H}^1
  (\mathbb{R})$.
\end{proof}

\begin{lemma}\label{lem:A1}
   We have $\lim_{L \rightarrow \infty} \limsup_{N \rightarrow \infty} A_1^{L, N} = 0$.
\end{lemma}

\begin{proof}
  We expand the squared $\dot{H}^{- 1}$ norm as
  \begin{align*}
    (A_1^{L, N})^2 & =\mathbb{E} \left[ \left\| \int_{\mathbb{R}} \varphi (x)
    W_1 (g_x^{L, N} (y_1, \cdot)) \diamond \phi_0^L (x) \dd x
    \right\|_{\dot{H}^{- 1}_{y_1} (\mathbb{R})}^2 \right]\\
    & = \int_{\mathbb{R}^2} \varphi (x) \varphi (x') \mathbb{E} \left[ \langle
    \phi_0^L (x) \diamond W_1 (g_x^{L, N} (y_1, \cdot)), \phi_0^L (x')
    \diamond W_1 (g_{x'}^{L, N} (y_1, \cdot)) \rangle_{\dot{H}^{- 1}_{y_1}
    (\mathbb{R})} \right] \dd x \dd x' .
  \end{align*}
  Integrating by parts the $W_1$ terms and taking into account the cancellations due to the partial Wick contractions, we get
  \begin{align*}
    & \mathbb{E} \left[ \langle \phi_0^L (x) \diamond W_1 (g_x^{L, N} (y_1,
    \cdot)), \phi_0^L (x') \diamond W_1 (g_{x'}^{L, N} (y_1, \cdot))
    \rangle_{\dot{H}^{- 1}_{y_1} (\mathbb{R})} \right]\\
    &\hspace{50pt}=\mathbb{E} [\phi_0^L (x) \phi_0^L (x')] \int_{\mathbb{R}} \langle
    g_x^{L, N} (y_1, y_2), g_{x'}^{L, N} (y_1, y_2) \rangle_{\dot{H}^{-
    1}_{y_1} (\mathbb{R})} \dd y_2\\
    &\hspace{50pt} \quad + \int_{\mathbb{R}^2} \mathbb{E} [(\dD_{y_3} \phi_0^L (x))
    (\dD_{y_2} \phi_0^L (x'))] \langle g_x^{L, N} (y_1, y_2), g_{x'}^{L, N}
    (y_1, y_3) \rangle_{\dot{H}^{- 1}_{y_1} (\mathbb{R})} \dd y_2 \dd
    y_3
  \end{align*}
  The second term can be written as
  \begin{align*}
    &\int_{\mathbb{R}^2} \mathbb{E} [(\dD_{y_3} \phi_0^L (x))
    (\dD_{y_2} \phi_0^L (x'))] \langle g_x^{L, N} (y_1, y_2), g_{x'}^{L, N}
    (y_1, y_3) \rangle_{\dot{H}^{- 1}_{y_1} (\mathbb{R})} \dd y_2 \dd
    y_3\\
    &\hspace{50pt} = \lambda^2 \mathbb{E} [\phi^L_0 (x) \phi^L_0 (x')] \int_{\mathbb{R}^2}
    \Theta^L_x (y_3) \Theta^L_{x'} (y_2) \langle g_x^{L, N} (y_1, y_2),
    g_{x'}^{L, N} (y_1, y_3) \rangle_{\dot{H}^{- 1}_{y_1} (\mathbb{R})}
    \dd y_2 \dd y_3,
  \end{align*}
  so letting $V^L (x, x') \assign \lambda^2 \varphi (x) \varphi (x')
  \mathbb{E} [_{} \phi_0^L (x) \phi_0^L (x')]$ we have
  \begin{align}    \label{eq:A1-decomposition} \nonumber 
    (A_1^{L, N})^2 & = \int_{\mathbb{R}^2} \dd x \dd x' V^L (x, x')
    \int_{\mathbb{R}} \langle g_x^{L, N} (y_1, y_2), g_{x'}^{L, N} (y_1, y_2)
    \rangle_{\dot{H}^{- 1}_{y_1} (\mathbb{R})} \dd y_2\\ \nonumber
    & \quad + \int_{\mathbb{R}^2} \dd x \dd x' V^L (x, x') \int_{\mathbb{R}^2}
    \Theta^L_x (y_3) \Theta^L_{x'} (y_2) \langle g_x^{L, N} (y_1, y_2),
    g_{x'}^{L, N} (y_1, y_3) \rangle_{\dot{H}^{- 1}_{y_1} (\mathbb{R})}
    \dd y_2 \dd y_3\\
    & = : A_{1, 1} + A_{1, 2} .
  \end{align}
  Let us consider first $A_{1, 2} = \int_{\mathbb{R}^2} V^L (x, x') \langle
  G^{L, N}_{x, x'}, G^{L, N}_{x', x} \rangle_{\dot{H}^{- 1} (\mathbb{R})}
  \dd x \dd x'$, which according to Lemma~\ref{lem:Gxx} can be bounded by
  \[
     A_{1, 2} \leqslant \int_{\mathbb{R}^2} V^L (x, x') \| G^{L, N}_{x, x'}
     \|_{\dot{H}^{- 1} (\mathbb{R})} \| G^{L, N}_{x', x} \|_{\dot{H}^{- 1}
     (\mathbb{R})} \dd x \dd x'  \lesssim L^{-1} \int_{\mathbb{R}^2} V^L (x, x') \dd x \dd x' \lesssim L^{-1}
  \]
  for all large $N$. We continue by estimating the term $A_{1, 1}$ in~(\ref{eq:A1-decomposition})
  which is bounded by
  \begin{equation}
    \label{eq:A11} A_{1, 1} \leqslant \int_{\mathbb{R}^2} \dd x \dd x'
    | V^L (x, x') | \int_{\mathbb{R}} \| g_x^{L, N} (\cdummy, y_2)
    \|_{\dot{H}^{- 1} (\mathbb{R})} \| g_{x'}^{L, N} (\cdummy, y_2)
    \|_{\dot{H}^{- 1} (\mathbb{R})} \dd y_2 .
  \end{equation}
  To treat the $\dot{H}^{- 1} (\mathbb{R})$ norms we argue again by duality, as in the proof of Lemma~\ref{lem:Gxx}.
  Let therefore $f \in C^{\infty}_c (\mathbb{R})$ and consider
  \begin{align}\label{eq:A11-expectation} \nonumber
    & \langle g_x^{L, N} (\cdummy, y_2), f \rangle_{L^2 (\mathbb{R})}\\ \nonumber
    &\hspace{15pt} = \int_{\mathbb{R}} \dd z (\rho^L_x (z) - \langle \rho^L_z, \sigma
    \rangle_{L^2 (\mathbb{R})}) (\rho^N \ast f) (z) \rho^N_z (y_2) - ((\rho^L
    \ast f) (x) \rho^L_x (y_2) - \langle (\rho^L \ast f) \rho^L_{y_2}, \sigma
    \rangle_{L^2 (\mathbb{R})})\\ \nonumber
    &\hspace{15pt} =\mathbb{E} [f (x - R^L_1 - R^N_2) \rho^N (x - y_2 - R^L_1) - f (y_2 -
    R^N_1 - R^N_2) \sigma (y_2 - R^N_1 - R^L_3)]\\
     &\hspace{15pt} \quad -\mathbb{E} [f (x - R^L_1) \rho^L (x - y_2) -
    f (y_2 - R^L_1 - R^L_2) \sigma (y_2 - R^L_1)],
  \end{align}
  where $R^L_i, R^N_i, Z$ are independent random variables as above. Now
  observe that by our assumptions on $\rho$
  \[ \mathbb{E} [\rho^N (x - y_2 - R^L_1) - \rho^L (x - y_2)] = (\rho^L \ast
     \rho^N) (x - y_2) - \rho^L (x - y_2) = 0 \]
  if $N$ is large enough, and similarly $\mathbb{E} [\sigma (y_2 - R^N_1 -
  R^L_3)] =\mathbb{E} [\sigma (y_2 - R^L_1)]$. So for large $N$ we can
  decompose the expectations in~(\ref{eq:A11-expectation}) as
  \begin{align*}
    \langle g_x^{L, N} (\cdummy, y_2), f \rangle_{L^2 (\mathbb{R})}
    & =\mathbb{E} [(f (x - R^L_1 - R^N_2) - f (x - R^L_1)) \rho^N (x - y_2 -
    R^L_1)]\\
    &\quad +\mathbb{E} [(f (x - R^L_1) - f (x)) (\rho^N (x - y_2 - R^L_1) - \rho^L
    (x - y_2))]\\
    &\quad -\mathbb{E} [(f (y_2 - R^N_1 - R^N_2) - f (y_2 - R^L_1 - R^L_2)) \sigma
    (y_2 - R^N_1 - R^L_3)]\\
    &\quad -\mathbb{E} [(f (y_2 - R^L_1 - R^L_2) - f (y_2)) (\sigma (y_2 - R^N_1 -
    R^L_3) - \sigma (y_2 - R^L_1))] .
  \end{align*}
  Bounding each term individually as in the proof of Lemma~\ref{lem:Gxx} we get
  \begin{align*}
    \langle g_x^{L, N} (\cdummy, y_2), f \rangle_{L^2 (\mathbb{R})} & \leqslant
    \| f \|_{\dot{H}^1 (\mathbb{R})} \times \Big(\mathbb{E} [| R^N_2 |^{1 / 2}
    \rho^N (x - y_2 - R^L_1)] \nobracket\\
    &\hspace{65pt} +\mathbb{E} [| R^L_1 |^{1 / 2} (\rho^N (x - y_2 - R^L_1) + \rho^L (x -
    y_2))]\\
    &\hspace{65pt} +\mathbb{E} [| R^N_1 + R^N_2 + R^L_1 + R^L_2 |^{1 / 2} \sigma (y_2 -
    R^N_1 - R^L_3)]\\
    &\hspace{65pt}  +\mathbb{E} [| R^L_1 + R^L_2 |^{1 / 2} (\sigma (y_2 - R^N_1 -
    R^L_3) + \sigma (y_2 - R^L_1))]\Big),
  \end{align*}
  which yields
  \begin{align*}
    \| g_x^{L, N} (\cdummy, y_2) \|_{\dot{H}^{- 1} (\mathbb{R})} &  \leqslant
    \mathbb{E} [| R^N_2 |^{1 / 2} \rho^N (x - y_2 - R^L_1)] +\mathbb{E} [|
    R^L_1 |^{1 / 2} (\rho^N (x - y_2 - R^L_1) + \rho^L (x - y_2))]\\
    &\quad +\mathbb{E} [| R^N_1 + R^N_2 + R^L_1 + R^L_2 |^{1 / 2} \sigma (y_2 -
    R^N_1 - R^L_3)] \\
    &\quad  +\mathbb{E} [| R^L_1 + R^L_2 |^{1 / 2} (\sigma (y_2 -
    R^N_1 - R^L_3) + \sigma (y_2 - R^L_1))] .
  \end{align*}
  By the same computation we get a similar bound for $\| g_{x'}^{L, N}
  (\cdummy, y_2) \|_{\dot{H}^{- 1} (\mathbb{R})}$, and plugging these back
  into~(\ref{eq:A11}) we generate a number of products between different
  expectations. Let us treat three prototypical cases: Writing $V^L (x)
  \assign | \lambda \varphi (x) | \mathbb{E} [_{} | \phi_0^L (x) |^2]^{1 /
  2}$, we have
  \begin{align*}
    & \int_{\mathbb{R}^2} \dd x \dd x' | V^L (x, x') | \int_{\mathbb{R}}
    \mathbb{E} [| R^L_1 |^{1 / 2} \rho^N (x - y_2 - R^L_1)] \mathbb{E} [|
    R^L_1 |^{1 / 2} \rho^N (x' - y_2 - R^L_1)] \dd y_2\\
    &\hspace{20pt} \leqslant \int_{\mathbb{R}^2} \dd x \dd x' V^L (x) V^L (x')
    \int_{\mathbb{R}} \mathbb{E} [| R^L_1 |^{1 / 2} \rho^N (x - y_2 -
    R^L_1)] \mathbb{E} [| R^L_1 |^{1 / 2} \rho^N (x' - y_2 - R^L_1)] \dd
    y_2\\
    &\hspace{20pt} = \int_{\mathbb{R}} \mathbb{E} [| R^L_1 |^{1 / 2} (\rho^N \ast V^L) (y_2
    + R^L_1)] \mathbb{E} [| R^L_1 |^{1 / 2} (\rho^N \ast V^L) (y_2 + R^L_1)]
    \dd y_2\\
    &\hspace{20pt} =\mathbb{E} \left[ | R^L_1 |^{1 / 2} | R^L_2 |^{1 / 2} \int_{\mathbb{R}}
    (\rho^N \ast V^L) (y_2 + R^L_1) (\rho^N \ast V^L) (y_2 + R^L_2) \dd y_2
    \right]\\
    &\hspace{20pt} \leqslant \mathbb{E} [| R^L_1 |^{1 / 2} | R^L_2 |^{1 / 2} \| (\rho^N \ast
    V^L) \|_{L^2 (\mathbb{R})}^2] \lesssim L^{- 1} (\| \rho^N \|_{L^1
    (\mathbb{R})} \| V^L \|_{L^2 (\mathbb{R})})^2 \lesssim L^{- 1},
  \end{align*}
  where we introduced a new independent copy $R^L_2$ of $R^L_1$ which is a
  trick that we will apply several times in the following. Another situation
  occurs if only one of the two expectations depends on $x$ (respectively
  $x'$), for example
  \begin{align*}
    & \int_{\mathbb{R}^2} \dd x \dd x' V^L (x) V^L (x')
    \int_{\mathbb{R}} \mathbb{E} [| R^N_2 |^{1 / 2} \rho^N (x - y_2 -
    R^L_1)] \mathbb{E} [| R^L_1 + R^L_2 |^{1 / 2} \sigma (y_2 - R^N_1 -
    R^L_3)] \dd y_2\\
    &\hspace{20pt} = \| V^L \|_{L^1 (\mathbb{R})} \int_{\mathbb{R}} \mathbb{E} [| R^N_2
    |^{1 / 2} (\rho^N \ast V^L) (y_2 + R^L_1)] \mathbb{E} [| R^L_1 + R^L_2
    |^{1 / 2} \sigma (y_2 - R^N_1 - R^L_3)] \dd y_2\\
    &\hspace{20pt} \simeq \mathbb{E} \left[ | R^N_2 |^{1 / 2} | R^L_4 + R^L_5 |^{1 / 2}
    \int_{\mathbb{R}} (\rho^N \ast V^L) (y_2 + R^L_1) \sigma (y_2 - R^N_6 -
    R^L_7) \dd y_2 \right]\\
   &\hspace{20pt}  \leqslant \mathbb{E} [| R^N_2 |^{1 / 2} | R^L_4 + R^L_5 |^{1 / 2} \|
    \rho^N \ast V^L \|_{L^2 (\mathbb{R})} \| \sigma \|_{L^2 (\mathbb{R})}]\\
    &\hspace{20pt} \lesssim N^{- 1 / 2} L^{- 1 / 2} \| \rho^N \|_{L^1 (\mathbb{R})} \| V^L
    \|_{L^2 (\mathbb{R})} \| \sigma \|_{L^2 (\mathbb{R})} \lesssim L^{- 1} .
  \end{align*}
  Finally, we have to handle the case where none of the expectations depend on $x$ or $x'$, for example
  \begin{align*}
    & \int_{\mathbb{R}^2} \dd x \dd x' V^L (x) V^L (x')
    \int_{\mathbb{R}} \mathbb{E} [| R^L_1 + R^L_2 |^{1 / 2} \sigma (y_2 -
    R^N_1 - R^L_3)] \mathbb{E} [| R^L_1 + R^L_2 |^{1 / 2} \sigma (y_2 -
    R^L_1)] \dd y_2\\
    &\hspace{20pt}= \| V^L \|_{L^1 (\mathbb{R})}^2 \mathbb{E} \left[ | R^L_1 + R^L_2 |^{1
    / 2} | R^L_4 + R^L_5 |^{1 / 2} \int_{\mathbb{R}} \sigma (y_2 - R^N_1 -
    R^L_3) \sigma (y_2 - R^L_4) \dd y_2 \right]\\
    &\hspace{20pt}\lesssim \mathbb{E} [| R^L_1 + R^L_2 |^{1 / 2} | R^L_4 + R^L_5 |^{1 / 2}
    \| \sigma \|_{L^2 (\mathbb{R})}^2] \lesssim L^{- 1} .
  \end{align*}
  In conclusion also $A_{1, 1}$ vanishes as first $N \rightarrow \infty$ and
  then $L \rightarrow \infty$, and this concludes the proof.
\end{proof}

\begin{lemma}\label{lem:C1}
  We have $\lim_{L \rightarrow \infty} \lim_{N \rightarrow \infty} C_1^{L, N} = 0$.
\end{lemma}

\begin{proof}
  Recall that
  \begin{align*}
    (C_1^{L, N})^2 & =\mathbb{E} \left[ \left\| \int_{\mathbb{R}} \varphi (x)
    \phi^L_0 (x) \int_{\mathbb{R}} g_x^{L, N} (y_1, y_2) \Theta^L_x (y_1)
    \dd y_1 \dd x \right\|_{\dot{H}^{- 1}_{y_2} (\mathbb{R})}^2
    \right]\\
    & = \int_{\mathbb{R}^2} V^L (x, x') \langle G^{L, N}_{x, x}, G^{L, N}_{x',
    x'} \rangle_{\dot{H}^{- 1} (\mathbb{R})} \dd x \dd x' .
  \end{align*}
  So by Lemma~\ref{lem:Gxx} we get directly $(C_1^{L, N})^2 \lesssim L^{- 1}$ for all large $N$, from where the convergence immediately follows.
\end{proof}

Lemma~\ref{lem:remainder} now follows by combining Lemma~\ref{lem:rLN-decomposition}, Lemma~\ref{lem:A1}, and Lemma~\ref{lem:C1}.

\subsubsection{Proof of Lemma~\ref{lem:constant process}}

Recall that
\[
  Q^L_t \assign \int_0^t \{ - \langle (u^L_s)^2 - \| \rho^L \|_{L^2(\R)}^2, \sigma
  \rangle_{L^2(\R)} + \| \rho^L \ast \sigma \|_{L^2(\R)}^2 + \lambda^{- 1} u_s (\rho^L
  \ast \partial_z \sigma) \} \dd s.
\]
By Corollary~\ref{cor:kpz drift hoelder} the first term on the right hand side converges in $L^p(\Omega; C^{3/4-}([0,T],\R))$ to $- \int_0^\cdot u^{\diamond 2}_s \dd s (\sigma)$ whenever $p\geqslant 1$ and $T > 0$. The convergence of the remaining terms is obvious, and overall we get
\[
  \lim_{L \to \infty} Q^L_t = Q_t \assign - \int_0^t u^{\diamond 2}_s \dd s (\sigma) + \| \sigma \|_{L^2(\R)}^2 t + \lambda^{-1} \int_0^t u_s(\partial_z \sigma) \dd s,
\]
where the convergence takes place in $L^p(\Omega; C^{3/4-}([0,T],\R))$.

\appendix
\section{The periodic case}\label{app:periodic}

For the periodic equation described Section~\ref{sec:periodic} most of the analysis works in the same way. The It\^o trick and the Kipnis-Varadhan inequality are shown using exactly the same arguments, and also the Gaussian analysis of Section~\ref{sec:gaussian} works completely analogously. We only have to replace all function spaces over $\R^n$ by the corresponding spaces over $\T^n$, say $L^2(\R^n)$ by $L^2(\T^n)$. The construction of the Burgers nonlinearity and the proof of its time-regularity also carry over to the periodic setting, although we have to replace the integrals over $\R^2$ in Fourier space by sums over $\Z^2$; but since those sums can be estimated by the corresponding integrals, we get the same bounds.

The first significant difference is in the construction of the integral. As discussed in Section~\ref{sec:map-SHE} any integral $Iu$ of $u \in \CS'(\R)$ is determined uniquely by its derivative $u$ and the value $Iu(\sigma)$ for some $\sigma \in \CS(\R)$ with $\int_\R \sigma(x) \dd x \neq 0$. The same is true on the circle, and here there is a canonical candidate for the function $\sigma$, namely the constant function $1$. So let $(u,\mathcal A) \in \CQ(W,\eta)$ be a pair of controlled processes, where $W$ is a periodic space-time white noise and $\eta$ a periodic space white noise. Let $\rho \in \CS (\mathbb{R})$ be an even function with $\hat{\rho} \in C^\infty_c (\mathbb{R})$ and such that $\hat{\rho} \equiv 1$ on a neighborhood of $0$ and define
\begin{equation}\label{eq:uL-def-periodic}
   u^L_t \assign \CF^{-1}_\T(\hat{\rho}(L^{-1}\cdot) \CF_\T u_t ) = \bar{\rho}^L \ast u_t, \qquad t \geqslant 0
\end{equation}
where $\CF_\T u(k) \assign \int_\T e^{2 \pi i k x} u(x) \dd x$ respectively $\CF^{-1}_\T \psi(x) \assign \sum_{k \in \Z} e^{2\pi i k x} \psi(k)$ denote the Fourier transform (respectively inverse Fourier transform) on the torus, $u \ast v(x) \assign \int_\T u(x-y) v(y) \dd y$ is the convolution on the torus, and $\bar{\rho}^L \assign \sum_{k \in \Z} L \rho(L(\cdot + k))$ is the periodization of $\rho^L \assign L \rho(L \cdot)$. For the last identity in~\eqref{eq:uL-def-periodic} we applied Poisson summation, see for example~\cite[Lemma~6]{bib:ebp}. We then integrate $u^L$ by setting
\[
   h^L_t \assign \CF^{-1}_\T (\CF_\T{\Theta} \CF_\T(u^L_t)) = \Theta \ast u^L_t = (\Theta \ast \bar{\rho}^L) \ast u_t =: \Theta^L \ast u_t, \qquad t \geqslant 0,
\]
where $\CF_\T{\Theta}(k) = \1_{k \neq 0} (2\pi i k)^{-1}$, which corresponds to
\[
   \Theta(x) = \1_{[-\frac{1}{2},0)}(x) (-x - \frac{1}{2}) + \1_{[0,\frac{1}{2})}(x) (-x + \frac{1}{2}) = - x - \frac{1}{2} + \1_{[0,\frac{1}{2})}(x), \qquad x \in [-\frac{1}{2}, \frac{1}{2}),
\]
or equivalently $\Theta(x) = \frac{1}{2} - x$ for $x \in [0,1)$. From the representation as Fourier multiplier it is obvious that $\partial_x (\Theta \ast u) = \Pi_0 u$, and since we assumed that $u^L_t = \Pi_0 u^L_t$ for all $t \geqslant 0$ we get $\partial_x h^L = u^L$. Writing $(\Theta \ast \bar{\rho}^L) \ast u_t(x) = u_t(\Theta^L_x)$ for $\Theta^L_x(y) \assign (\Theta \ast \bar{\rho}^L)(x - y)$, we get from the fact that $u$ is a strong stationary solution of the periodic Burgers equation that
\begin{align*}
   \dd h^L_t (x) & = u_t (\Delta_z \Theta^L_x) \dd t + \dd  \mathcal{A}_t (\Theta^L_x) + \sqrt{2} \dd W_t (-\partial_z \Theta^L_x) \\
   & = \Delta_x h^L_t(x) + \dd  \mathcal{A}_t (\Theta^L_x) + \sqrt{2} \dd W_t (\partial_x \Theta^L_x)
\end{align*}
and $\dd [h^L(x)]_t = 2 \| \partial_x \Theta^L_x \|_{L^2(\T)}^2 \dd t$. From the expression $\Theta(x) =  \frac{1}{2} - x$ for $x \in [0,1)$ we see that $\partial_x \Theta_x = \delta - 1$, where $\delta$ denotes the Dirac delta, and therefore $\partial_x \Theta^L_x = \bar{\rho}^L_x - 1$ for $\bar{\rho}^L_x(y) = \bar{\rho}^L(x - y)$. So setting $\phi^L_t (x) \assign e^{\lambda h^L_t (x)}$ we have
\begin{align*}
  \dd \phi^L_t (x) & = \phi^L_t (x) \left( \lambda \dd h^L_t (x) +
  \frac{1}{2} \lambda^2 \dd [ h^L (x) ]_t \right)\\
  & = \lambda \phi^L_t (x) \left( \Delta_x h^L_t(x) + \dd  \mathcal{A}_t (\Theta^L_x) + \sqrt{2} \dd W_t (\partial_x \Theta^L_x) + \lambda \| \bar{\rho}^L_x - 1 \|_{L^2(\T)}^2 \dd t \right),
\end{align*}
and since $\lambda \phi^L_t (x) \Delta_x h^L_t (x) = \Delta_x \phi^L_t (x) - \lambda^2\phi^L_t (x) (\partial_x h^L_t (x))^2$ we get
\begin{align*}
  \dd \phi^L_t (x) & = \Delta_x \phi^L_t (x) \dd t + \sqrt{2} \lambda \phi^L_t (x) \dd W_t (\partial_x \Theta^L_x) \\
  &\quad +\lambda^2 \phi^L_t
  (x) (\lambda^{- 1} \dd \mathcal{A}_t (\Theta^L_x) - ((u_t^L (x))^2 - \| \bar{\rho}^L_x - 1 \|_{L^2(\T)}^2) \dd t) \\
  & = \Delta_x \phi^L_t (x) \dd t + \sqrt{2} \lambda \phi^L_t (x) \dd W_t (\bar{\rho}^L_x) + \lambda^2 \dd R^L_t(x) + \lambda^2 K^L \phi^L_t(x) \dd t + \lambda^2 \phi^L_t(x) \dd Q^L_t \\
  &\quad - \sqrt{2} \lambda \phi^L_t(x) \dd W_t(1) - 2\lambda^2 \phi^L_t(x) \dd t,
\end{align*}
where we expanded the $L^2(\T)$ norm and defined
\[ R^L_t (x) \assign \int_0^t \phi^L_s (x) \{ \lambda^{- 1} \dd
   \mathcal{A}_s (\Theta^L_x) - \Pi_0((u_s^L (x))^2) \dd s - K^L \dd s \} \]
for a constant $K^L$ to be determined, and
\[
  Q^L_t \assign \int_0^t \{ - \langle (u^L_s)^2 - \| \bar{\rho}^L \|_{L^2(\R)}^2, 1
  \rangle_{L^2(\T)} + 1 \} \dd s.
\]
From here on the proof is completely analogous to the non-periodic setting provided that we establish the following three lemmas.

\begin{lemma}\label{lem:remainder-periodic}
   We have for all $T > 0$, $p > 2$ and all $\varphi \in
  C^{\infty} (\mathbb{T})$
  \[ \lim_{L \rightarrow \infty} (\mathbb{E} [\| R^L (\varphi) \|_{p -
     \tmop{var} ; [0, T]}^2] +\mathbb{E} [\sup_{t \leqslant T} | R^L_t
     (\varphi) |^2]) = 0. \]
\end{lemma}

\begin{lemma}\label{lem:KL-periodic}
   The constant $K^L$ converges to $\lambda^2/12$ as $L \to \infty$.
\end{lemma}

\begin{lemma}\label{lem:constant process-periodic}
  For all $T > 0$ the process $(Q^L_t)_{t \in [0, T]}$ converges in probability in $C^{3 / 4 -}([0, T], \mathbb{R})$ to the zero quadratic variation process
 \[
   Q_t \assign - \int_0^t u^{\diamond 2}_s \dd s (1) + t , \qquad t \in [0,T].
\]
\end{lemma}

To prove these lemmas we follow the argumentation in Section~\ref{sec:convergence}. Here the kernel $g^{L,N}_x$ takes the form
\begin{equation}
  g_x^{L, N} (y_1, y_2) \assign  \int_{\mathbb{T}} \dd z (\bar{\rho}^L_x (z) - 1) \bar{\rho}^N_z (y_1) \bar{\rho}^N_z \left( y_2 \right) - (\bar{\rho}^L_x (y_1) \bar{\rho}^L_x (y_2) - \langle \bar{\rho}^L_{y_1} \bar{\rho}^L_{y_2}, 1 \rangle_{L^2 (\T)}),
\end{equation}
and as in Section~\ref{sec:convergence} we see that we should choose
\[
   K^{L, N}_x \assign \lambda^2 \int_{\mathbb{T}^2} g_x^{L, N} (y_1, y_2) \Theta^L_x (y_1) \Theta^L_x (y_2) \dd y_1 \dd y_2.
\]
The proof of Lemma~\ref{lem:KL-periodic} is not a trivial modification of the one of Lemma~\ref{lem:KL}, so we provide the required arguments.

\begin{proof}[Proof of Lemma~\ref{lem:KL-periodic}]
   Sending $N \to \infty$ and using that $\Theta^L_x(y) = \Theta^L(x-y)$ and $\bar{\rho}^L_x(y) = \bar{\rho}^L(x-y)$ we get
   \[
      \lim_{N \to \infty} K^{L,N}_x = K^L \assign \lambda^2 \int_{\T}  \dd z (\bar{\rho}^L_x (z) - 1)  \Theta^L_{x} (z) \Theta^L_{x} (z) - \lambda^2 (\langle \Theta^L, \bar{\rho}^L\rangle_{L^2(\T)}^2  - \langle (\bar{\rho}^L \ast \Theta^L)_x^2, 1 \rangle_{L^2 (\T)}),
   \]
   and the first term on the right hand side is $\lambda^2 (-1/3) \int_\T \dd z \partial_{z} (\Theta^L_x(z)^3) = 0$.  Moreover, also
   \[
     \langle \Theta^L, \bar{\rho}^L\rangle_{L^2(\T)} = \int_\T \Theta^L(y) \bar{\rho}^L(y) \dd y = \int_\T \Theta^L(y) (\bar{\rho}^L(y) - 1) \dd y = \int_\T \frac{1}{2} \partial_y (\Theta^L(y)^2) \dd y = 0,
   \]
   and therefore we remain with
   \begin{align*}
      \lambda^{-2} K^L & = \langle (\bar{\rho}^L \ast \Theta^L)_x^2, 1 \rangle_{L^2 (\T)} = \| \bar{\rho}^L \ast \Theta^L \|_{L^2(\T)}^2 = \sum_{k \in \Z} |\CF_\T (\bar{\rho}^L \ast \Theta^L)(k)|^2 \\
      & = \sum_{k \in \Z \setminus\{0\}} \frac{|\rho(L^{-1}k)|^2}{(2\pi k)^2} = \frac{1}{2\pi^2} \sum_{k=1}^\infty \frac{|\rho(L^{-1}k)|^2}{k^2} \\
      & \xrightarrow[]{L \to \infty} \frac{1}{2\pi^2} \sum_{k=1}^\infty \frac{1}{k^2}= \frac{1}{12},
   \end{align*}
   which concludes the proof.
\end{proof}

The rest of the proof is completely analogous to the non-periodic case. Let us just point out that if $f \in C^\infty(T)$, then
\[
   \int_\T (f(x - y) - f(x)) \bar{\rho}^L(y) \dd y = \E[f(x - R^L) - f(x)],
\]
where $R^L$ is a random variable with density $\bar{\rho}^L$, and that
\[
  f(x-R^L_1) - f(x) = \int_{x(\mathrm{mod} 1)}^{(x-R^L_1) (\mathrm{mod} 1)} \partial_y f(y) \dd y,
\]
and therefore the same line of argumentation as in Section~\ref{sec:convergence} yields Lemma~\ref{lem:remainder-periodic} and Lemma~\ref{lem:constant process-periodic}. From here we follow the same steps as in the proof of Theorem~\ref{thm:burgers-uniqueness} to establish Theorem~\ref{thm:burgers-uniqueness-periodic}.

\end{document}